\documentclass[11pt,reqno,a4paper]{amsart}

\usepackage[utf8]{inputenc}
\usepackage[english]{babel}
\usepackage{bookmark}
\usepackage{amsmath}
\usepackage{amsthm}
\usepackage{mathtools}

\usepackage{stmaryrd}
\usepackage{bm}
\usepackage{bbm}
\usepackage{amssymb}
\usepackage[dvipsnames]{xcolor}
\usepackage{graphicx}
\usepackage{tcolorbox}
\usepackage{float}

\usepackage{caption}
\usepackage{subcaption}
\usepackage{listings}
\usepackage{enumitem}
\usepackage[
  hmarginratio={1:1},               
  vmarginratio={1:1},               
  textwidth=400pt,                  
  heightrounded                     
]{geometry}
\usepackage{csquotes}
\MakeOuterQuote{"}

\usepackage{tikz}
\usetikzlibrary{positioning,
decorations.pathreplacing,
patterns,arrows,math}
\usepackage{pgfplots}
\pgfplotsset{compat=1.15}
\usepackage{mathrsfs}

\usepackage{anyfontsize}            

\DeclareMathOperator*{\argmax}{argmax}

\begin{document}

\theoremstyle{plain}
\newtheorem{theorem}{Theorem}
\newtheorem{proposition}[theorem]{Proposition}
\newtheorem{lemma}[theorem]{Lemma}
\newtheorem{definition}[theorem]{Definition}
\newtheorem{remark}[theorem]{Remark}
\newtheorem{example}[theorem]{Example}
\numberwithin{equation}{section}
\newtheorem{assumption}{Assumption}

\newcommand{\N}{\mathbb{N}}
\newcommand{\Z}{\mathbb{Z}}
\newcommand{\Q}{\mathbb{Q}}
\newcommand{\R}{\mathbb{R}}
\newcommand{\C}{\mathbb{C}}
\newcommand{\E}{\mathbb{E}}
\renewcommand{\S}{\mathbb{S}}
\renewcommand{\L}{\mathcal{L}}
\renewcommand{\H}{\mathcal{H}}
\newcommand{\eps}{\varepsilon}
\newcommand{\Cc}{\mathcal{C}}
\newcommand{\spt}{\mathrm{spt}}
\newcommand{\diam}{\mathrm{diam}}
\newcommand{\inc}{\subseteq}
\newcommand{\isEquivTo}[1]{\underset{#1}{\sim}}
\newcommand{\mres}{\mathbin{
\vrule height 1.6ex depth 0pt width
0.13ex\vrule height 0.13ex depth 0pt width 1.3ex}}
\newcommand{\Lag}{\mathrm{Lag}}
\newcommand{\RLag}{\mathrm{RLag}}
\DeclarePairedDelimiter{\Iint}{\llbracket}{\rrbracket}

\date{\today}
\title{Unidimensional semi-discrete
partial optimal transport}
\author[Adrien Cances]{Adrien Cances}
\address{Adrien Cances (\href{mailto:adrien.cances@universite-paris-saclay.fr}{\tt adrien.cances@universite-paris-saclay.fr}), Université Paris-Saclay, Laboratoire de mathématiques d’Orsay, ParMA, 91405, Orsay, France}
\author[Hugo Leclerc]{Hugo Leclerc}
\address{Hugo Leclerc  (\href{mailto:hugo.leclerc@universite-paris-saclay.fr}{\tt hugo.leclerc@universite-paris-saclay.fr}), Université Paris-Saclay, Laboratoire de mathématiques d’Orsay, ParMA, 91405, Orsay, France}

\maketitle
\begin{sloppypar}

\begin{abstract}
We study the semi-discrete formulation of one-dimensional
partial optimal transport with quadratic cost,
where a probability density is partially transported
to a finite sum of Dirac masses of smaller total mass.
This problem arises naturally in applications such as
risk management, the modeling of crowd motion,
and sliced partial transport algorithms for point cloud registration.
Unlike higher-dimensional settings, the dual functional
in the unidimensional case exhibits reduced regularity.
To overcome this difficulty, we introduce
a regularization procedure based on thickening
the density along an auxiliary dimension.
We prove that the maximizers of the regularized dual problem
converge to those of the original dual problem,
with quadratic rate in the introduced thickness.
We further provide a numerical scheme that leverages
the regularized functional, and we validate our analysis
with simulations that confirm the quadratic convergence rate.
Finally, we compare the semi-discrete and fully discrete settings,
demonstrating that our approach offers both improved stability
and computational efficiency for unidimensional
partial transport problems.
\end{abstract}

\section{Introduction}

The standard definition of optimal
transport requires
both measures to have the same (finite) mass.
In~\cite{caffarelli2010free},
Caffarelli and McCann introduced the notion
of \textit{partial} optimal transport
\footnote{Several authors (including Caffarelli and McCann)
use the term \textit{optimal partial transport}, but both
expressions are found in the literature.},
which lifts this assumption and asks
how to move a prescribed fraction of mass
from the first measure to the second one,
in the cheapest possible way.
As a result, partial optimal transport allows to compare
two measures of \textit{different masses}.
There is an extensive literature on
the case where both the target and
source measures are absolutely continuous.
For the quadratic cost,
existence and uniqueness of solutions,
as well as regularity of the free boundaries
--- the boundaries of the \textit{active regions},
on which sits the respective amounts of mass
of the source and target measures that are
actually transported --- have been studied
in the seminal work~\cite{caffarelli2010free},
and in~\cite{figalli2010optimal,indrei2013free}.
A few years later,
Chen and Indrei study the free boundary regularity
for a more general cost function
in~\cite{chen2015regularity},
and, in~\cite{davila2016dynamics},
Dávila and Kim inspect the evolution of
the free boundaries and the monotonicity
of the optimal potential
as the fraction of mass
that must be transported varies
At last, in 2018, the papers of Igbida and
Nguyen\cite{igbida2018optimal,igbida2018augmented}
address the partial optimal transport problem for a general cost and
with generic source and target measures.
The authors introduce a PDE of
Monge-Kantorovich type which enables them
to extend the uniqueness and monotonicity results
to this more general framework, and is also
convenient for numerical analysis and computations.
It is interesting to note that partial optimal transport
is a special case of \textit{unbalanced} optimal transport,
obtained by using total variation for the penalization terms.
We refer to~\cite[Section~2.3]{chizat2017unbalanced}
for a detailed description of this result.

\subsection{Sliced partial optimal transport}
As its balanced version,
partial optimal transport is a computationally
expensive problem, especially in high dimension.
This has led several
authors~\cite{rabin2012wasserstein,
bonneel2015sliced}
to work on \textit{sliced} partial optimal transport:
the two measures are projected on all lines passing
through the origin, and the
corresponding partial transport costs
are integrated.
In the spirit of the sliced-Wasserstein distance,
which was defined by Peyré
et al.~\cite{rabin2012wasserstein}
and Rabin et al.~\cite{bonneel2015sliced},
this definition leverages the fact that the
problem is easier to solve in dimension one.
However, while one-dimensional transport is trivial
in the sense that for empirical measures,
it amounts to a simple sort,
one-dimensional partial transport is much more complex.

\subsection{Motivation for
one-dimensional partial optimal transport}
In 2019, Bonneel et al.~\cite{bonneel2019spot}
have proposed an algorithm to solve
the fully-discrete case.
More precisely,
their algorithm tackles the case where the
two measures are (unweighted) sums of
Dirac masses of $\R$, of the form
$\sum_{x\in X}\delta_{x}$,
with sets of points $X,\tilde X$ of
different cardinalities $N\neq\tilde N$.
The goal is to injectively assign every point
of the smaller point cloud to a point
of the larger one.
The method can then be used to compute an
approximation of the sliced partial optimal transport
cost between two (unweighted)
sets of points of different cardinalities.
More recently, Schmitzer et al.~\cite{bai2023sliced}
used a dual formulation to create an algorithm
for the same problem, although slightly more
general, since a tunable Lagrange parameter
allows for leaving certain points of the
source out of the mapping.
Both methods have a quadratic worst-case
time complexity, but in practice they achieve
quasi-linear complexity.
One of the applications highlighted
in~\cite{bonneel2019spot,bai2023sliced}
is an improved version
of the Iterative Closest Point (ICP) algorithm
for point cloud registration.
The latter problem consists in estimating,
among a certain class
of transformations, the one that best aligns
a given point cloud to another one (of different
cardinality).
While the ICP algorithm is quite efficient
when considering rigid transformations
(rotation and translation), it behaves in practice
much less well for similarities
(rotation, translation and scaling)
because the scaling factor tends to
go to zero as the iterations progress.
Bonneel et al. showed that a sliced-OPT-based
ICP helps tackle this issue 
by providing
an injective matching at each iteration,
while avoiding the intractability of
full-fledged optimal transport on large problems
in dimension $d\geq 2$.
Schmitzer et al., on the other hand,
developed an ICP-type algorithm
which harnesses sliced partial optimal transport
to be more robust to outliers, whose proportion
is assumed to be known.

\subsection{Semi-discrete, one-dimensional, partial optimal transport}
In this paper, we focus on the semi-discrete
setting. Specifically, the problem consists in
transporting part of a given probability density
to a (positively) weighted sum of Dirac measures
whose total mass is less than one, so as to
minimize the quadratic transport cost.
This problem can be applied to risk
management, see~\cite{ennaji2024robust},
where the goal is to solve a
partial multimarginal optimal transport problem
in which each one of the $D$ marginals is
a univariate probability density $\rho_j$.
If one searches an approximate solution
as a discrete measure
$\mu = \sum_{i=1}^N \alpha_i \delta_{y_i}$,
where $y_{1},\dots,y_{N}\in\R^{D}$ and where
$\alpha_{i},\dots,\alpha_{N}$ are non-negative
weights summing to less than one,
then one-dimensional
semi-discrete partial optimal transport
is a natural way to penalize the constraints.
Indeed, the smaller the optimal partial transport cost
from $\rho_j$ to $\pi^j_{\#}\mu
= \sum_{i=1}^N \alpha_i \delta_{y_i^j}$,
the closer the $j$\textsuperscript{th} marginal
of $\mu$ is to being dominated by $\rho_j$.
This type of penalization already appears
in~\cite{leclerc2020lagrangian}
for solving a crowd motion problem
via Lagrangian discretization, although
in dimension two.
The crowd is modeled as a probability density,
constrained to be bounded by one to account for the
incompressibility of individuals when
densely packed, see for
instance~\cite{maury2011handling}.
When passing to a discrete measure
for numerical issues,
this congestion constraint is dealt with
using a penalty term proportional
to the partial optimal transport cost between
the discrete probability measure and
the (unnormalized) Lebesgue measure on the domain.
This ensures that the discrete measure
in question is \textit{close} to a
density bounded by one.

\subsection{Precise setting}
Let us now introduce a few notations to better
define the problem addressed in this work.
We denote $\rho$ the probability density on $\R$
and $\sum_{i=1}^N \alpha_i \delta_{y_i}$
the discrete measure,
where the $y_i$ are pairwise distinct
points of $\R$
and where the $\alpha_i>0$ sum to
$\|\alpha\|_1\in(0,1)$.
The goal is to minimize
the (quadratic) transport cost
$\int |x-y|^2 d\gamma(x,y)$
over all measures $\gamma$ on $\R\times\R$
whose first marginal
is bounded above by $\rho$ on all Borel sets
and whose second marginal is $\mu$.
By standard duality arguments,
the problem is equal to its dual
\begin{equation}
\label{eq:unregularized-problem}
\max_{\psi\in\R^N} \left\{
\int_{\R} \min(0,\min_{1\leq i\leq N}
[|x-y_i|^2-\psi_i])d\rho(x)
+ \sum_{i=1}^N\alpha_i\psi_i
\right\},
\end{equation}
which is of particular interest since the
optimized dual variable $\psi$ lives
in \textit{finite dimension}.
This kind of formulation was already
used in~\cite{leclerc2020lagrangian}
by Leclerc and co-authors,
for semi-discrete optimal transport in dimension two or three.
However, in dimension one,
the problem is curiously much harder to solve
numerically.
We interpret this numerical difficulty
as a consequence of the lack of regularity
of the dual functional in dimension one,
compared to dimension two or more (replacing
absolute values with Euclidean norms), see
Theorem~\ref{thm:regularity-multiD}
and Example~\ref{ex:irregularity-1D}.
Indeed, whereas in the latter case
the functional is of class
$\Cc^2$ on a large open containing
the maximizers, in the former case the
functional is only $\Cc^1$, and its
second derivative has singularities.
In order to circumvent this difficulty,
we minimize a regularized functional,
which in fact corresponds
to solving a two-dimensional transport problem,
where the one-dimensional
density $\rho$ is thickened
by a width $2\eps$ along a second
dimension.
Thanks to the $\Cc^2$ regularity
of the new functional, the Hessian matrix
is well-defined.
The new problem reads
\begin{equation}
\label{eq:regularized-problem}
\max_{\psi\in\R^N}
- \eps^2\int_{\R} f^*(\eps^{-2}\max_{1\leq i\leq N}
(\psi_i-|x-y_i|^2))d\rho(x)
+ \sum_{i=1}^N\alpha_i\psi_i
\end{equation}
where
\begin{equation}
f^*(t) =
\begin{cases}
0 \quad &\mathrm{if} \quad t<0, \\
\frac{2}{3}t^\frac{3}{2} \quad &\mathrm{if}
\quad 0\leq t\leq 1, \\
t-\frac{1}{3} \quad &\mathrm{if} \quad t>1.
\end{cases}
\end{equation}
Our main theoretical result is the following theorem,
which guarantees that the
regularized problem converges at a
fast rate to the initial one as the parameter
$\eps$ goes to zero.
Moreover, Example~\ref{ex:quadratic-rate-is-tight}
shows that our result is sharp.
\begin{theorem}
\label{thm:intro}
Let $\psi^\eps\in\R^N$
be the maximizer
of Problem~\eqref{eq:regularized-problem}.
Under some regularity assumptions on the density,
we have
\begin{equation}
\label{eq:quadratic-rate-intro}
\|\psi^\eps-\psi^*\|
\lesssim \eps^2
\end{equation}
where $\psi^*\in\R^N$ is the maximizer
of Problem~\eqref{eq:unregularized-problem}
and where $\lesssim$ hides a constant
depending only on $\rho$, on
the $y_i$ and on the $\alpha_i$.
\end{theorem}
The proof harnesses standard tools
of semi-discrete optimal transport, and spectral graph
theory in order to show that, in dimension
$d\geq 2$ and assuming the probability density
has convex, compact support, the dual functional
of semi-discrete partial optimal transport is
\textit{strictly} concave on some subset containing the solution.
In particular, this implies that the respective
solutions $\psi^\eps$ of the regularized problems
are unique.
We emphasize that an important feature of
partial optimal transport is
its equivalence to a \textit{balanced}
version of the problem (where the two
measures have equal mass),
up to introducing a fictitious point and
extending the cost function to zero for this
particular point. This corresponds to imagining
that the \textit{inactive part} of
the source measure --- the part that is not sent
to the target measure --- is actually getting
transported ``for free'' to the auxiliary point.
This point of view was introduced
by Caffarelli and McCann
in~\cite[Section~2]{caffarelli2010free},
in the case of absolutely continuous
source and target measures, but the equivalence
still holds in the more general
framework of generic probability measures.
As a consequence for semi-discrete partial transport,
strict concavity of the
dual functional amounts to connectedness
of the (undirected) graph with vertices
$\{1,\dots,N,\infty\}$
and weights given by the Hessian of
the functional
(the diagonal coefficients weigh the edges
between $\infty$ and the other vertices).
The reason why the dual variable $\psi$
is unique in semi-discrete partial optimal transport
(assuming the density has connected support)
is that the ``hidden'' component of $\psi$
associated to the fictitious point
is implicitly set to zero. This eliminates
the underdeterminacy of balanced
semi-discrete optimal transport, for which adding the same
constant to each component of the dual
variable yields the same value.

To show quadratic convergence of
$\psi^\eps$ towards $\psi^*$, we first
differentiate in $\eps$ the optimality
conditions satisfied by $\psi^\eps$
to get an ODE of the form
$D^2\mathcal{K}^\eps(\psi^\eps)\dot\psi^\eps
+\partial_\eps[\nabla\mathcal{K}^\eps](\psi^\eps)$.
The Hessian at the maximizer can be
bounded below by a positive-definite matrix
independent of $\eps$,
once again thanks to spectral graph theory,
and more specifically to Laplacian matrices and
a careful analysis of the second derivatives
of $\mathcal{K}^\eps$.
As for the mixed derivative
$\partial_\eps[\nabla\mathcal{K}^\eps]$,
it can be bounded above quite straightforwardly
by $\eps$ up to a multiplicative constant,
and the quadratic rate follows by
integrating the resulting bound on $\dot\psi^\eps$.

\subsection{Outline}
After a brief overview of optimal transport,
we introduce its semi-discrete partial
counterpart, in its
primal and dual formulations.
We establish that the dual functional
is of class $\mathcal C^2$ for
dimensions greater than one,
whereas its Hessian has singularities
in the case of dimension one.
Next, we present the regularized problem and prove
Theorem~\ref{thm:intro} by introducing an ODE
that is satisfied by the family of solutions
$\eps\mapsto\psi^\eps$.
The final section is dedicated to numerics.
After a brief description of the algorithm,
we provide simulation results that
illustrate the quadratic convergence rate
predicted by Theorem~\ref{thm:intro},
and apply our algorithm to a semi-discrete
version of the sliced partial transport matching.
At last, we compare the discrete-discrete
versus semi-discrete settings of unidimensional
optimal partial transport, both in terms of accuracy
and execution time, via a benchmark of Boonnel's
algorithm and ours.

\section{Semi-discrete optimal transport
and the case of partial transport}

We first introduce the Wasserstein distance
(of order $2$)
for probability measures on $\R^d$,
as well as its dual formulation and the specific
framework of the semi-discrete case.

\begin{definition}
For two probability measures
with finite second moment
$\rho, \mu \in \mathcal P_2(\R^d)$,
the \textnormal{$2$-Wasserstein distance}
is defined as
\begin{equation}
\label{eq:Wasserstein-distance}
W_2(\rho, \mu) =
\left(
\inf_{\gamma \in \Gamma(\rho, \mu)}
\int_{\R^d \times \R^d}
\|x-y\|^2 \; d\gamma(x,y)
\right)^{\frac{1}{2}},
\end{equation}
where $\Gamma(\rho, \mu)$ denotes
the set of \textnormal{couplings}
(or \textnormal{transport plans})
of $\rho$ and $\mu$, that is, probability measures
on $\R^d\times \R^d$ with first marginal $\rho$
and second marginal $\mu$.
\end{definition}



The celebrated Brenier
theorem~\cite{brenier1991polar} states that under
mild assumptions --- for example if the source
measure $\rho$ is absolutely continuous ---
the optimal transportation does
\textit{not} split mass.
More precisely, any optimal plan is induced
by an (optimal transport) map $T\in L^2(\rho)$,
which is uniquely defined $\rho$-almost everywhere
and (naturally) sends $\rho$ to $\mu$,
in the sense that $T_\#\rho=\mu$.
The notation $T_\#\rho$ stands for
the \textit{pushforward} of $\rho$ by $T$,
defined by $[T_\#\rho](B) = \rho(T^{-1}(B))$
for every measurable set $B$, and we say that
$T$ \textit{pushes $\rho$ forward} to $T_\#\rho$.
Any optimal transport plan thus writes
$\gamma = (\mathrm{Id},T)_\#\rho$, and the mass
at $\rho$-almost every point $x$
is sent to a unique point $y=T(x)$.

\subsection{Kantorovich duality}
The optimal transport problem~\eqref{eq:Wasserstein-distance}
is an infinite-dimensional
linear program, and
the Fenchel-Rockafellar theorem allows to
establish strong duality:
if $\rho$ and $\mu$ both have finite second moment,
the problem is equal to its dual
\begin{align}
\label{eq:Wasserstein-dual}
W_2(\rho, \mu)^2 &=
\sup \left\{
\int_{\R^d} \varphi \; d\rho
+ \int_{\R^d} \ \; d\mu :
\varphi,\psi \in \mathcal C_b(\R^d),
\varphi \oplus \psi \leq c
\right\}
\\
\label{eq:Wasserstein-dual-only-psi}
&= \sup_{
\psi \in \Cc_b(\R^d)
}
\left\{\int_{\R^d} \psi^c \; d\rho
+ \int_{\R^d} \psi \; d\mu \right\}
\end{align}
where $\varphi \oplus \psi$
maps $(x,y)\in\R^d\times\R^d$ to
$\varphi(x) + \psi(y)$,
$c$ is the quadratic cost $c(x,y)=\|x-y\|^2$,
and $\psi^c$ denotes
the \textit{$c$-transform} of
$\psi$, defined by
\begin{equation}
\label{eq:generic-c-transform}
\psi^c(x) =
\inf_{y\in \R^d}
c(x,y) - \psi(y).
\end{equation}
The interested reader may refer
to~\cite[Chapter~5]{villani2008optimal}
for a detailed study of the Kantorovich duality.
In this work, we will often refer to the dual variable
$\psi$ as the \textit{potential}, whether or not
it is optimal
in~\eqref{eq:Wasserstein-dual-only-psi}.
Note that most authors use this terminology
for optimal $\psi$ only.



\subsection{Semi-discrete transport}
In the semi-discrete case,
we consider an absolutely continuous
probability measure
$\rho \in
\mathcal P^{\mathrm{ac}}_2(\R^d)$
with finite second moment, and
a discrete probability measure
$\mu = \sum_{i=1}^N \alpha_i
\delta_{y_i}$,
where the $y_i \in \R^d$
are pairwise distinct
(without loss of generality)
and where the $\alpha_i>0$ sum to one.
With a slight abuse of notation,
throughout this work
we also denote $\rho$
for the density of $\rho$.
Using the fact that $\mu$ has finite support,
we identify the potential with a vector
$\psi\in\R^N$ and
the dual problem reads
$\sup_{\psi \in \R^N} \mathcal K(\psi)$,
where the Kantorovich functional
$\mathcal K : \R^d\to\R$
and the \textit{Laguerre cells}
are respectively
defined by
\begin{equation}
\label{eq:dual-functional}
\mathcal K(\psi) =
\sum_{i=1}^N \int_{
\Lag_i(\psi)}
(\|x-y_i\|^2 - \psi_i)
\; d\rho(x)
+ \sum_{i=1}^N \alpha_i \psi_i
\end{equation}
and
\begin{equation}
\label{eq:Lag-cells}
\Lag_i(\psi)
= \{x\in \R^d : \|x-y_i\|^2 - \psi_i
\leq \|x-y_j\|^2 - \psi_j,
\forall j \in \Iint{1,N}\}.
\end{equation}

The Laguerre cells form a tessellation of
the space --- that is, a cover with Lebesgue
negligible pairwise intersection
--- into $N$ convex polytopes,
each cell $\Lag_i(\psi)$ being sent to the
position $y_i$ of its corresponding Dirac mass.
Indeed, the intersection of any two Laguerre cells
belongs to a hyperplane orthogonal to the segment
joining the two related Dirac masses, and is
thus of Lebesgue measure zero.
Since $\rho$ is absolutely continuous, the set of
points belonging to more than one cell is
$\rho$-negligible, and $\psi$ induces a transport map
defined $\rho$-almost everywhere by
$T_\psi|_{\Lag_i(\psi)} \equiv y_i$.
The corresponding transport plan is $\gamma =
\sum_{i=1}^N(\rho\mres\Lag_i(\psi))
\otimes\delta_{y_i}$.
As a result, in the dual formulation
mass cannot be split, and by strong duality
the optimal plan is induced by a potential,
hence by a map in $L^2(\rho)$.

\begin{remark}
The convexity of the Laguerre cells is very
specific to the quadratic cost $c(x,y)=\|x-y\|^2$,
since for this cost the sublevel sets of
$c(\cdot,y_i)-c(\cdot,y_j) =
\langle\cdot,y_j-y_i\rangle
- \|y_j\|^2 + \|y_i\|^2$
are precisely half-spaces.
For non-quadratic costs, the functions
$c(\cdot,y_i) - c(\cdot,y_j)$ are not linear
in $x$ anymore and so we cannot have convexity of
both the sublevel and superlevel sets.
\end{remark}

\begin{remark}
\label{rem:c-transform-of-vector}
In semi-discrete optimal transport, by a slight abuse of
notation, we also denote $\psi^c$
the $c$-transform of a vector $\psi\in\R^N$,
which is defined by
$\psi^c(x) = \min_{1\leq i\leq N}
\|x-y\|^2 - \psi_i$.
Of course, it naturally corresponds to the
$c$-transform of the function
$\tilde\psi:\R^d\to\R\cup\{-\infty\}$
defined by $\tilde\psi(y_i)=\psi_i$
and $\tilde\psi(y)=-\infty$ for
$y\not\in\{y_1,\dots,y_N\}$.
\end{remark}

\subsection{Partial transport}
\label{subsec:partialtransport}

To compare two finite measures with
potentially different masses,
we use a partial Wasserstein cost,
which naturally extends the standard
Wasserstein distance.

\begin{definition}
For two (positive) measures
$\rho,\mu \in \mathcal{M}_{+,2}(\R^d)$
with finite second moment and such that
$\rho(\R^d) \geq \mu(\R^d)$, the
\textnormal{partial $2$-Wasserstein cost}
is defined as
\begin{equation}
\label{eq:partial-OT}
\mathcal{T}_{\geq}(\rho,\mu) =
\left(
\inf_{\gamma \in \Gamma_{\max}(\rho,\mu)}
\int_{\R \times \R}
\|x-y\|^2 d\gamma(x,y)
\right)^{\frac{1}{2}}
\end{equation}
where $\Gamma_{\max}(\rho,\mu)$ denotes the set
of measures $\gamma \in \mathcal
M(\R^d \times \R^d)$ of mass
equal to that of $\mu$,
whose second marginal is $\mu$, and whose first
marginal is \textnormal{dominated by $\rho$}
in the sense that
$\gamma(A\times\R^d) \leq \rho(A)$
for all measurable $A$.
The index \textnormal{max}
corresponds to the fact that
$\Gamma_{\max}(\rho,\mu)$
is the set of subcouplings
which have maximum mass, in the sense that
as much mass as possible
is transported from $\rho$ to $\mu$.
\end{definition}

Since what used to be an equality constraint
on the first marginal of $\gamma$ is
now an inequality constraint,
the associated Lagrange multiplier $\varphi$
must be non-positive, and the dual problem reads
\begin{align}
\mathcal{T}_{\geq}(\rho,\mu)^2 &=
\sup\left\{
\int_{\R^d} \varphi \; d\rho
+ \int_{\R^d} \psi \; d\mu :
\varphi,\psi \in \Cc(\R^d),
\varphi\oplus\psi\leq c,\varphi\geq 0
\right\} \\
&= \sup_{\psi \in \Cc(\R^d)}\left\{
\int_{\R^d} \psi^{\tilde c} \; d\rho
+ \int_{\R^d} \psi \; d\mu
\right\}
\end{align}
where $\psi^{\tilde c}$ denotes the
\textit{$\tilde c$-transform} of $\psi$,
defined by
$\psi^{\tilde c}(x) = \min\{\psi^c(x),0\}$

\subsection{Semi-discrete partial transport}

We now make the assumption that the source
density $\rho \in \mathcal{P}^{\mathrm{ac}}(\R^d)$
is compactly supported and of
mass normalized to one,
and we take the target measure to be a discrete
measure of mass strictly less than one.
Namely,
$\mu = \sum_{i=1}^N \alpha_i \delta_{y_i}$
where the $y_i \in \R^d$ are pairwise distinct
and where the $\alpha_i>0$ sum
to some $\|\alpha\|_1 \in (0,1)$.
The potential is once again identified with a vector
$\psi\in\R^N$, and for a fixed $\psi$,
the optimal $\varphi$ is now the
$\psi^{\tilde c}$-transform of $\psi$, defined by
\[
\psi^{\tilde c}(x) = \min\left\{0,
\min_{1\leq i\leq N}\|x-y_i\|^2-\psi_i\right\}.
\]
The Kantorovich dual thus writes
$\sup_{\psi \in (0,\infty)^N}
\mathcal K(\psi)$,
where the dual functional
$K : (0,+\infty)\to\R$
and the \textit{restricted Laguerre cells}
are respectively defined by
\begin{equation}
\label{eq:dual-functional-SDOPT}
\mathcal K(\psi) =
\sum_{i=1}^N \int_{\RLag_i(\psi)}
(\|x-y_i\|^2 - \psi_i) \; d\rho(x)
+ \sum_{i=1}^N \alpha_i \psi_i
\end{equation}
and
\begin{equation}
\label{eq:RLag-SDOPT-1D}
\RLag_i(\psi)
= \Lag_i(\psi)
\cap \overline{B}_{y_i}(\sqrt{\psi_i}).
\end{equation}
Note that strict positivity of each component
of $\psi$ is a necessary condition for the
the restricted Laguerre cells to all be non-empty.
Also, depending on the potential $\psi$,
the restricted Laguerre cells do not necessarily
(and will most often not) cover $\spt\rho$.
When there is ambiguity,
we will refer to the $\Lag_i(\psi)$
as the \textit{unrestricted} Laguerre cells.
We now prove the existence of a maximizer via
a standard concavity argument.

\begin{proposition}
\label{prop:C1-regularity}
The functional
$\mathcal K : (0,\infty)^N \to \R$
is concave and of class $\Cc^1$,
with gradient
\begin{equation}
\label{eq:optimal-condition}
\nabla \mathcal{K}(\psi) = \alpha - G(\psi),
\end{equation}
where
$G_i = \rho(\RLag_i)$.
Therefore,
$\psi$ is a maximizer of $\mathcal{K}$
if and only if every restricted Laguerre cell
has the same mass as its corresponding Dirac
mass in the target measure:
\begin{equation}
\label{eq:optimality-conditions-SDOPT}
\mathcal \psi \in \argmax \mathcal K
\quad \Longleftrightarrow \quad
G_i(\psi) = \alpha_i, \; \forall
i\in \Iint{1,N}.
\end{equation}
Moreover, there exists a maximizer.
\end{proposition}
\begin{proof}
Thanks to equations~\eqref{eq:dual-functional-SDOPT}
and~\eqref{eq:RLag-SDOPT-1D}, we can write
\begin{equation}
\label{eq:dual-functional-SDOPT-no-RLag}
\mathcal K(\psi) =
\sum_{i=1}^N \int_{\Lag_i(\psi)}
\min\{\|x-y_i\|^2-\psi_i, 0\} \; d\rho(x)
+ \sum_{i=1}^N \alpha_i \psi_i,
\end{equation}
and the second sum is linear in $\psi$
(with constant gradient $\alpha$)
so we focus on the first sum.
Let $\psi, \psi' \in (0,\infty)^N$.
On $\Lag_i(\psi')$ we have
$\min\{\|x-y_i\|^2-\psi'_i, 0\}
\leq \min\{\|x-y_j\|^2-\psi'_j, 0\}$
for all $j$, so since for each potential
the (unrestricted)
Laguerre cells are pairwise disjoint and
cover $\spt\rho$, we have
\begin{align*}
&\sum_{i=1}^N \int_{\Lag_i(\psi')}
\min\{\|x-y_i\|^2-\psi'_i, 0\} d\rho(x) \\
&\leq
\sum_{j=1}^N \int_{\Lag_j(\psi)}
\min\{\|x-y_j\|^2-\psi'_j, 0\} d\rho(x)
= \sum_{j=1}^N \int_{\RLag_j(\psi)}
(\|x-y_j\|^2-\psi'_j) d\rho(x) \\
&=
\sum_{j=1}^N \int_{\RLag_j(\psi)}
(\|x-y_j\|^2-\psi_j) d\rho(x)
+ \sum_{j=1}^N \rho(\RLag_j(\psi))
(\psi_j - \psi_j').
\end{align*}
Adding the linear terms of
$\mathcal{K}(\psi')$ yields
$\mathcal{K}(\psi') \leq \mathcal{K}(\psi)
+ \langle\alpha-G(\psi),\psi'-\psi\rangle$,
where $G_i(\psi)=\rho(\RLag_i(\psi))$,
and it follows that
$\mathcal{K} : (0,\infty)^N \to \R$
is concave and that its superdifferential
at $\psi$ contains $\alpha-G(\psi)$.
Furthermore, by dominated convergence
the $\rho$-measure of each restricted Laguerre cell
is continuous in $\psi$, and we deduce
that $\mathcal{K}$ is indeed differentiable
with continuous gradient
$\nabla\mathcal{K} = \alpha-G$.
To prove the existence of a maximizer,
take a maximizing sequence $(\psi^n)$ and
suppose by contradiction that it is not bounded.
Then, up to extraction of a subsequence, there
exists an index $i$ such that $(\psi^n_i)$
monotonically converges to $-\infty$ or $+\infty$.
In the first case, the corresponding restricted
Laguerre cell is eventually empty, which is a
contradiction ; in the second case, the complement
$\R^d\setminus\cup_{i=1}^N\RLag_i(\psi^n)=
\R^d\setminus\cup_{i=1}^N
\overline{B}_{y_i}(\sqrt{\psi^n_i})$
of the union of the cells has $\rho$-measure
converging to zero, which is also a contradiction,
since summing the optimality
conditions~\eqref{eq:optimality-conditions-SDOPT}
yields $1-\sum_{i=1}^N G_i(\psi)=
1-\sum_{i=1}^N\alpha_i > 0$.
\end{proof}

\begin{remark}
\label{rem:auxiliary-point}
Unlike in the the balanced case, the functional of
semi-discrete \textnormal{partial} optimal transport
is \textnormal{not invariant} up to addition of an identical
constant to each component of the potential.
This is due to the target and source measures
having different mass.
However, one can retrieve this property by
introducing a fictitious point $y_\infty$
to which the \textnormal{inactive} mass of $\rho$
will be sent, and by defining a new cost function
\[
\underline{c}(x,y_i) =
\begin{cases}
\|x-y_i\|^2 \quad &\mathrm{if} \quad
i \in\Iint{1,N}, \\
0 \quad &\mathrm{if} \quad
i=\infty.
\end{cases}
\]
In this setting, the potential has an additional
component $\psi_\infty$, and the
$i$\textsuperscript{th} \textnormal{restricted}
Laguerre cell of the partial transport setting is now
the $i$\textsuperscript{th} \textnormal{unrestricted}
Laguerre cell for these new potential
$\underline{\psi} = (\psi_1,\dots,\psi_N,\psi_\infty)$
and cost function $\underline{c}$. Indeed,
the constraint indexed by $j=\infty$ states that
$\|x-y_i\|^2 - \psi_i \leq 0 - \psi_\infty$,
i.e.~that $x$ is in the closed ball of center $y_i$
and of radius $\sqrt{\psi_i-\psi_\infty}$.
The additional Laguerre cell is sent to the
auxiliary point $y_\infty$, and is the closure
of the complement of the $N$ open balls
$B_{y_i}(\sqrt{\psi_i-\psi_\infty})$.
The reason for which the dual functional
$\mathcal{K}$ in~\eqref{eq:dual-functional-SDOPT}
--- or equivalently
in~\eqref{eq:dual-functional-SDOPT-no-RLag} ---
is not invariant up to addition of the same constant
to each component of the potential is that we
implicitly set $\psi_\infty=0$ in the balanced
transport formulation that we have just described.
Note that in this formulation, the
additional Laguerre cell is $\textnormal{not}$
convex, and in fact not even connected.
\end{remark}

In this article, we will mainly deal with
the one-dimensional case.
Let us emphasize that semi-discrete partial
transport in 1D is somewhat \textit{degenerate}
compared to higher dimensions.
This is due to the fact that the boundary
of each restricted Laguerre cell contains only
\textit{finitely} many points --- two points
exactly, the left and right extremities --- and that,
in some cases, their velocity
is not linear in the (arbitrarily small)
variation of the potential.
As will be apparent in section~\ref{subsec:SDOPT-1D},
this is due to the fact that these extremities are
respectively defined as a maximum and a minimum,
and that the maximum or minimum of two smooth
quantities is in general not smooth when the
latter coincide. We will thus be lead to
consider the \textit{unilateral} partial
derivatives of each cell's boundaries.

\begin{theorem}
\label{thm:regularity-multiD}
Fix $d\geq 2$ and suppose that $\rho$ is
compactly supported and absolutely continuous,
with continuous density on the support.
In addition, suppose that the intersection
of $\spt\rho$ with any sphere centered
at one of the $y_i$ is $\rho\H^{d-1}$-negligible,
as well as the intersection of $\spt\rho$
with any hyperplane intersecting its interior
and orthogonal to some segment $[y_i,y_j]$.
Then the functional $\mathcal{K}$ is $\Cc^2$
on the convex open set $\mathcal{D}$
of potentials such that
each restricted Laguerre cell has strictly positive
$\rho$-measure, as well as the complement
of their union.
The second derivatives are given by
\begin{equation}
\label{eq:non-diag-coeff-Hessian-multiD}
\partial^2_{\psi_i,\psi_j}
\mathcal{K}(\psi) =
\frac{1}{2\|y_i-y_j\|}
\int_{\RLag_i(\psi) \cap
\RLag_j(\psi)} \rho(x)
d\H^{d-1}(x)
\end{equation}
for $i\neq j$, and
\begin{equation}
\label{eq:diag-coeff-Hessian-multiD}
\partial^2_{\psi_i,\psi_i}
\mathcal{K}(\psi) =
- \frac{1}{2\sqrt{\psi_i}}\int_{
\RLag_i(\psi) \cap
S_{y_i}(\sqrt{\psi_i})} \rho(x)
d\H^{d-1}(x)
- \sum_{j\neq i} \partial^2_{\psi_i,\psi_j}
\mathcal{K}(\psi).
\end{equation}
In particular, $\mathcal{K}$
is strictly concave on $\mathcal{D}$
and the optimal potential is unique.
\end{theorem}

\begin{remark}
\label{rem:support-assumption-multiD}
The assumptions on the support of $\rho$
guarantee for instance that
$\partial_{\psi_i,\psi_j}^2\mathcal{K}$
does not jump when the spherical part of
the boundary of $\RLag_i$
crosses the support.
A simple case for which the assumptions on $\spt\rho$
are met is when the support of $\rho$ is a
bounded convex polygon.
\end{remark}

In dimension $d\geq 2$, the Kantorovich
functional~\eqref{eq:dual-functional-SDOPT}
is of class $\Cc^2$ whenever the density $\rho$
has a bounded polytope as its support and is continuous
on this set,
while it is only of class $\Cc^1$
in the unidimensional setting.
Before dealing with the case of dimension one,
let us state the precise result we derived for
the multidimensional case,
the proof of which is in
Appendix~\ref{sec:appendix-multiD}.

\subsubsection*{Edges of a cell}
The key notion to understand (at least conceptually)
the $\Cc^2$ regularity of the dual
functional in dimension greater than one is that
of \textit{edges} of a (restricted Laguerre) cell.
We call \textit{edge} of a cell
any intersection of the cell with two other cells,
or with one other cell and the inactive part
(the auxiliary cell introduced
in Remark~\ref{rem:auxiliary-point}).
The edges of a given cell are precisely
the points of its boundary whose respective
velocities are ill-defined. Indeed, if a point
belongs to an edge, then the two (or more)
facets it belongs to have different velocity
vectors as we vary the potential $\psi$.
Roughly speaking, in dimension $d \geq 2$,
the differentiability of the mass of a cell
with respect to the potential is due to the fact
that its edges are of dimension $d-2$, and
are thus not ``seen'' by the $(d-1)$-dimensional
Hausdorff measure.

\subsection{The unidimensional case}
\label{subsec:SDOPT-1D}

As mentioned above, in the one-dimensional case
the gradient of $\mathcal{K}$ is not differentiable.
However, it does have a weaker kind of regularity,
as stated in the following lemma,
the proof of which is reported to
Appendix~\ref{sec:appendix-uniqueness-SDOPT-1D}.
For convenience, we assume (without loss of generality)
that the positions of the Dirac masses are
labeled in increasing order:
$y_1 < y_2 < \dots < y_N$.
In this case, provided every (unrestricted)
Laguerre cell has non-empty interior, we have
$\Lag_i = [z_{i-1},z_i]$, where the positions
\begin{equation}
\label{eq:z_i}
z_i(\psi) = \frac{y_i + y_{i+1}}{2}
- \frac{\psi_{i+1} - \psi_i}{2(y_{i+1} - y_i)}
\end{equation}
are the boundaries between consecutive Laguerre cells.
Naturally, we set $z_0 \equiv -\infty$
and $z_N \equiv +\infty$.
The restricted Laguerre cells can then
be written
$\RLag_i = [a_i, b_i]$,
where
\begin{equation}
\label{eq:extremities-of-RLag-1D}
\begin{cases}
a_i(\psi) = \max\{z_{i-1}(\psi),
&y_i - \sqrt{\psi_i}\},\\
b_i(\psi) = \min\{z_{i}(\psi),
&y_i + \sqrt{\psi_i}\}.
\end{cases}
\end{equation}

We will make the following assumption
on the source measure.

\begin{assumption}
\label{a:density-SDOPT-1D}
The probability measure $\rho\in\mathcal{P}(\R)$
is absolutely continuous, its support is a bounded
interval, and on this interval its density $\rho$
is both continuous and bounded away from zero
and infinity, i.e. there exists finite strictly
positive constants $\rho_{\min},\rho_{\max}$ such that
$\rho_{\min} \leq \rho(x) \leq \rho_{\max}$
for every $x\in\spt\rho$.
\end{assumption}

\begin{lemma}
\label{lem:technical-SDOPT-1D}
Suppose that $\rho$ satisfies
Assumption~\ref{a:density-SDOPT-1D}.
Denote $\mathcal{D}$ the open set of potentials
for which every restricted Laguerre cell has nonzero
$\rho$-measure, as well as the complement of their union.
At each $\psi \in \mathcal D$, the function $G$
has well-defined unilateral directional derivatives
$\partial_v^+ G(\psi) := 
\lim_{t \to 0^+} (G(\psi + tv)-G(\psi))/t$
with respect to any vector $v\in\R^N$.
Moreover, any such derivative can be written
$\partial_v^+ G(\psi) = H(\psi,v)v$ for some
tridiagonal, weakly diagonally dominant
matrix $H(\psi,v)$.
\end{lemma}

Each $H(\psi,v)$ can in addition be chosen
to be symmetric, so that it is block diagonal,
with all blocks being symmetric, tridiagonal,
irreducible, weakly diagonally dominant matrices ;
see Lemma~\ref{lem:block-decomposition} in
Appendix~\ref{sec:appendix-uniqueness-SDOPT-1D}.
At last, we show in
Lemma~\ref{lem:strictly-dominant-coefficient}
of the same section that each of these blocks has
at least one strictly diagonally dominant coefficient,
and the matrix $H(\psi,v)$ is thus non-singular
by Taussky's theorem, see for
instance~\cite[Corollary~6.2.27]{horn2012matrix}.
This allows for the following uniqueness result.

\begin{proposition}
If $\rho$ satisfies Assumption~\ref{a:density-SDOPT-1D},
then the dual functional
\mbox{$\mathcal K : (0,\infty)^N \to \R$}
admits a unique maximizer $\psi^*$,
and this maximizer naturally belongs to the set
$\mathcal D$ defined in
Lemma~\ref{lem:technical-SDOPT-1D}.
\end{proposition}

\begin{proof}
It it straightforward to see that any maximizer
must be in $\mathcal{D}$.
Suppose that $\psi^{(1)}$ and $\psi^{(2)}$
are two distinct maximizers.
Then, by concavity, every convex combination
of these is also a maximizer, so setting
$v=\psi^{(2)}-\psi^{(1)}$ we have
$\partial_v^+G(\psi^*)=0$.
But the latter quantity writes $H(\psi^{(1)},v)v$,
where $H(\psi^{(1)},v)$ is a non-singular matrix,
so $v$ must be zero and we have
$\psi^{(1)}=\psi^{(2)}$.
\end{proof}

\begin{example}
\label{ex:irregularity-1D}
A concrete example of non-differentiability
can be constructed as follows.
Take two Dirac masses, respectively
at $y_1 = 0$ and $y_2 = 1$, and
$\rho$ the Lebesgue measure on $[0,1]$.
The two unrestricted Laguerre cells intersect at
$z_1(\psi) = \frac{1}{2}(1-\psi_2+\psi_1)$.
Consider the potential $\psi^0 = 
(\frac{1}{4},\frac{1}{4})$,
for which $z_1(\psi^0)$ coincides with
$\sqrt{\psi^0_1}=\frac{1}{2}$,
and let $\psi^t = \psi^0 + t(1,0)$, $t\in\R$.
The derivatives of $z_1(\psi^t)$ and
$\sqrt{\psi^t_1}$ at $t=0$ are respectively
$\frac{1}{2}$ and $1$.
Thanks to equation~\eqref{eq:extremities-of-RLag-1D},
at $t=0$ the quantity $b_1(\psi^t)$
has left-hand derivative $-1$ and
right-hand derivative $\frac{1}{2}$.
Since $a_1 \equiv 0$ and $\rho \equiv 1$
on $[0,1]$, we deduce that
$G_1(\psi^t)$ has the same left-hand
and right-hand derivatives at $t=0$ as $b_1(\psi^t)$,
which shows that $G$ is not differentiable at $\psi^0$.
\end{example}

In order to numerically solve (balanced)
semi-discrete optimal transport, Mérigot
\textit{et al.}~\cite{kitagawa2019convergence}
introduced a damped Newton algorithm;
in particular they use a variable step size approach
(backtracking line search) to make sure that,
throughout the descent, the Laguerre cells all
have a mass greater than a given threshold.
To adapt this algorithm to the partial version
of the problem, we have added a natural
extra constraint: the complement
of the union of the Laguerre cells must also
have mass greater than the chosen threshold.
Unfortunately, the resulting algorithm is
not as efficient as in the balanced case:
the convergence crucially depends on the
initial point and the gradient descent
often involves long plateau phases,
with the algorithm occasionally getting stuck.
These difficulties are likely linked
to the non-$\mathcal C^2$ nature of
the dual functional.
In order to avoid this kind of issues,
we introduce a regularization of the partial
optimal transport problem,
which we detail in the section below.

\section{Regularization of semi-discrete partial
transport in dimension one}

The partial optimal transport problem~\eqref{eq:partial-OT}
can be written
\begin{equation}
\label{eq:unregularized-SDOPT-1D}
\inf_{\gamma \in \Gamma_{\max}(\rho, \mu)}
\int_{\R\times\R} |x-y|^2 d\gamma(x,y)
= \inf_{\sigma \in \mathcal{M}_+(\R)}
W_2^2(\sigma,\mu) + F_0(\sigma),
\end{equation}
where the functional
$F_0 : \mathcal{M}_+(\R)\to\R\cup\{+\infty\}$
defined by
\[
F_0(\sigma) =
\begin{cases}
\int_{\R} \chi_{[0,1]}(\frac{d\sigma}{d\rho}) d\rho
\; &\mathrm{if} \; \sigma \ll \rho, \\

+\infty \; &\mathrm{otherwise},
\end{cases}
\]
ensures that $\sigma$ is dominated by $\rho$.
Here, $\chi_{[0,1]}$ denotes the
characteristic function of the interval $[0,1]$
in the sense of convex analysis.
In this formulation,
the optimal $\sigma$ corresponds to the
\textit{active part} of $\rho$, that is, the
part of $\rho$ that actually gets sent to
$\mu$ in the partial transport problem.
In the semi-discrete framework, the optimal $\sigma$
is simply the restriction of $\rho$ to the union
of the restricted Laguerre cells.

\subsection{The regularized functional}
In order to regularize
Problem~\eqref{eq:unregularized-SDOPT-1D},
we may replace $F_0$ by a functional $F$ of the form
\begin{equation*}
F(\sigma) =
\begin{cases}
\int_{\R} f(\frac{d\sigma}{d\rho}) d\rho
\; &\mathrm{if} \; \sigma \ll \rho, \\
+\infty \; &\mathrm{otherwise},
\end{cases}
\end{equation*}
where $f : \R \to [0,+\infty]$
is a convex function approximating $\chi_{[0,1]}$
(in a sense which will be clarified below),
and try to solve
\begin{equation}
\label{eq:regularized-SDOPT-1D-primal}
\inf_{\sigma \in \mathcal{M}_+(\R)}
W_2^2(\sigma,\mu) + F(\sigma).
\end{equation}

\begin{assumption}
\label{a:regularization-function-f}
Throughout the rest of this subsection,
we make the following assumptions on $f$,
which guarantee that $F$ be both convex and
lower semi-continuous for
the weak convergence of measures,
and that the domain of $F$ is included
in the set of positive measures.
The fact that $f$ is strictly convex, with
domain $[0,1]$ and subderivatives bounded below
by $1$ guarantees that $f^*$ has full domain and is
convex $\Cc^1$, non-decreasing, superlinear,
with non-decreasing derivative upper bounded by $1$.
\begin{enumerate}
    \item $f : \R \to [0,+\infty]$
    is proper, lower semi-continuous, and
    strictly convex on its domain.
    \item $\mathrm{dom}f = [0,1]$.
    \item The subderivatives of $f$ are all at least $1$.
\end{enumerate}
\end{assumption}

By classical arguments, the dual problem
of~\eqref{eq:regularized-SDOPT-1D-primal} reads
$\sup_{\psi \in \R^N}\mathcal K^f(\psi)$, where
the regularized functional $\mathcal{K}^f:\R^d\to\R$
is defined by
\begin{equation}
\label{eq:functional-K^f}
\mathcal K^f(\psi) =
- \sum_{i=1}^N \int_{\Lag_i(\psi)}
f^*(\psi_i - |x-y_i |^2) d\rho(x)
+ \sum_{i=1}^N \alpha_i \psi_i.
\end{equation}
Moreover, the following proposition --- which holds
in particular when $\rho$ and $f$ respectively satisfy
Assumption~\ref{a:density-SDOPT-1D} and
Assumption~\ref{a:regularization-function-f} ---
states that the regularized functional $\mathcal{K}^f$
is concave and indeed twice continuously differentiable on
a large domain containing its maxima.


\begin{proposition}
\label{prop:functional-reg-SDOPT-1D}
Suppose that $\rho$ is an absolutely continuous
probability measure and that $f$ satisfies
Assumption~\ref{a:regularization-function-f}.
Then, the functional $\mathcal K^f$ is both concave
and of class $\Cc^1$ on $\R^N$, with gradient
\begin{equation}
\label{eq:grad-K^f}
\nabla \mathcal K^f(\psi) = \alpha - G^f(\psi),
\end{equation}
where the quantity
\begin{equation}
\label{eq:regularized-mass-of-Lag-cell}
G^f_i(\psi) = \int_{\Lag_i(\psi)}
(f^*)'(\psi_i - |x-y_i|^2) \; d\rho(x)
\end{equation}
will be referred to as the \textnormal{regularized mass}
of the $i$\textsuperscript{th }Laguerre cell.

If in addition $\rho$ has compact, convex support,
with continuous density on this set,
and $(f^*)'$ is piecewise $\Cc^1$,
then $\mathcal{K}^f$ is of class $\Cc^2$
on the open set $\tilde{\mathcal{D}}$ of potentials
such that all Laguerre cells have strictly positive
$\rho$-measure..
On this set, the Hessian matrix $D^2 \mathcal K^f(\psi)$
is (symmetric) tridiagonal with coefficients
\begin{equation*}
[D^2 \mathcal K^f(\psi)]_{i,i+1} =
\frac{1}{2(y_{i+1}-y_i)} (f^*)'(
\psi_i - |z_i(\psi)-y_i|^2)
\rho(z_i(\psi))
\end{equation*}
and
\begin{align*}
[D^2 \mathcal K^f(\psi)]_{i,i} &=
-\int_{\Lag_i(\psi)}
(f^*)''(\psi_i - |x-y_i|^2) \; d\rho(x) \\
&\qquad
- [D^2 \mathcal K^f(\psi)]_{i,i-1}
- [D^2 \mathcal K^f(\psi)]_{i,i+1}.
\end{align*}
\end{proposition}


\begin{remark}
Since the domain of $f$ is the set of non-negative
numbers and $f$ is increasing on this set, we have
$f^*(t) \geq -f(0)$ for every $t\in\R$,
with equality if $t\leq 0$.
The regularizing functions we will define in
section~\ref{subsec:a-family-of-regularizing-functions}
will be zero at $x=0$, so their respective
convex conjugates will be non-negative,
with support in $[0,+\infty)$.
\end{remark}

\begin{proof}


The second sum in the right-hand side
of~\eqref{eq:functional-K^f} is linear in $\psi$,
so we only consider the first sum, which writes
$\int_\R f^*(\max_{1\leq i\leq N}\psi_i-|x-y_i|^2)d\rho(x)$.
Since $f^*$ is continuous and $\rho$ has compact
support, this quantity is differentiable
thanks to dominated convergence,
and its first derivatives are given
by~\eqref{eq:regularized-mass-of-Lag-cell},
Continuity of these derivatives is a direct
consequence of the dominated convergence theorem.
To check that $\mathcal{K}^f$ is convex,
we conveniently write the first integral as
$\int_\R f^*(-\psi^c)d\rho$.
The $c$-transform is convex in $\psi$ as a
minimum of affine linear functions, so since
$f^*$ is convex and non-decreasing, the integral
in question is a convex function of $\psi$,
and concavity of $\mathcal{K}^f$ follows.

Let us now suppose that $(f^*)'$ is piecewise $\Cc^1$.
Thanks to Leibniz integral rule,
we get the desired expression for the
second derivatives of $\mathcal{K}^f$,
and these are continuous thanks to the
additional assumption that the density
$\rho$ is continuous on its support.
\end{proof}

\subsection{A family of regularizing functions}
\label{subsec:a-family-of-regularizing-functions}
Let us now introduce a convenient
family of functions $(f_\eps)_{\eps>0}$
approximating $\chi_{[0,1]}$ and
satisfying all the assumptions in
Proposition~\ref{prop:functional-reg-SDOPT-1D}.
Define $f_\eps(t) = \eps^2 f(t)$, where
\begin{equation}
\label{eq:definition-of-f*}
f(t) =
\begin{cases}
\frac{1}{3} t^3
\quad &\mathrm{if} \quad 0 \leq t \leq 1, \\
+\infty \quad &\mathrm{otherwise},
\end{cases}
\end{equation}
so that, taking the derivatives to be zero
on the negative half-line,
\begin{align}
\label{eq:f*}
f^*(t) =
\begin{cases}
\frac{3}{2}t^{\frac{2}{3}}
\quad &\mathrm{if} \quad 0 < t < 1, \\
t - \frac{1}{3}
\quad &\mathrm{if} \quad t > 1, \\
0 \quad &\mathrm{if} \quad
\mathrm{otherwise,}
\end{cases}
\end{align}
\begin{align}
\label{eq:derivatives-of-f*}
\begin{aligned}[c]
(f^{*})'(t) =
\begin{cases}
\min\{\sqrt{t},1\}
\quad &\mathrm{if} \quad
t > 0, \\
0 \quad &\mathrm{otherwise},
\end{cases}
\end{aligned}
\;\;
\begin{aligned}[c]
(f^{*})''(t) =
\begin{cases}
\frac{1}{2\sqrt{t}}
\quad &\mathrm{if} \quad
0 < t < 1, \\
0 \quad &\mathrm{otherwise}.
\end{cases}
\end{aligned}
\end{align}


The derivatives of $f_\eps$ then read
\[
f_{\eps}^*(t) = \eps^2 f^*(\eps^{-2}t), \qquad
(f_{\eps}^*)'(t) = (f^*)'(\eps^{-2}t), \qquad
(f_{\eps}^*)''(t) = \eps^{-2}(f^*)''(\eps^{-2}t).
\]
Figure~\ref{fig:derivative-of-f*}
shows the graph of $(f^*_\eps)'$ for different values
of $\eps$, as well as their pointwise limit
$\mathbbm 1_{(0,+\infty)}$ as $\eps\to 0$.
Note that replacing $(f^*)'$ by this pointwise limit
in the right-hand side
of~\eqref{eq:regularized-mass-of-Lag-cell}
yields the mass of the \textit{restricted} Laguerre cell
(with respect to $\rho$). Similarly, taking
$f^*=\lim_{\eps\to0}f_\eps^*
=\mathbbm{1}_{\geq0}\mathrm{Id}$
in~\eqref{eq:functional-K^f}, we recover
the unregularized functional for the uni-dimensional
partial transport problem.

\subsubsection*{Notations}
For convenience, we write write
$\mathcal K^\eps$, $G^\eps$
for $\mathcal K^{f^\eps}$, $G^{f^\eps}$
respectively.

\begin{figure}
\centering
\includegraphics[width=0.5\textwidth]{
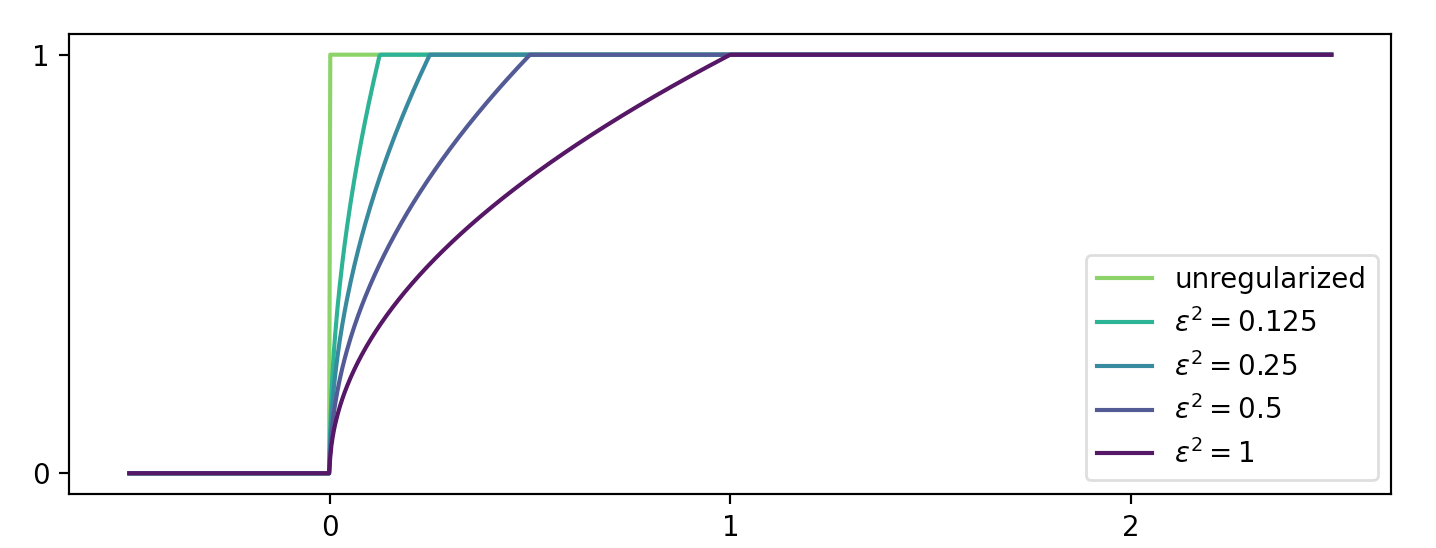}
\caption{Graphs of
$(f_\eps^*)'$ for
different values of
$\eps$,
as well as the limit graph
when $\eps\to 0$,
which corresponds to the
unregularized problem.}
\label{fig:derivative-of-f*}
\end{figure}

\subsection{Two-dimensional interpretation}
The choice of $f^\eps$ was motivated by
the fact that the corresponding regularized problem
is a two-dimensional version of the original one.
Indeed, define $\rho^\eps$ the probability measure
on $\R^2$ with density
$\rho^\eps(x) =
(2\eps)^{-1}\rho(x^1)
\mathbbm 1_{|x^2|<\eps}$
for $x=(x^1,x^2)$,
whose support is the rectangle
$\spt\rho \times [-\eps,\eps]$.
The discrete measure $\mu$ is naturally
extended to $\R^2$ by
identifying $y_i$ to $z_i=(y_i,0)$.
Since the $z_i$ are all aligned
on the first axis,
a straightforward computation shows that the
Laguerre cells of the
two-dimensional problem are
the vertical strips
$\Lag_i^{(2)}(\psi)
= \Lag_i(\psi) \times \R$
and that the corresponding
restricted Laguerre cells are
$\RLag_i^{(2)}(\psi) =
(\Lag_i(\psi) \times \R)
\cap \overline{B}^{(2)}_{z_i}(\sqrt{\psi_i})$.
As a result, we can write
$\RLag_i^{(2)}(\psi)$
as the following disjoint union
of vertical segments
\[
\RLag_i^{(2)}(\psi) =
\bigsqcup_{x^1\in\Lag_i(\psi)}
\{x^1\} \times
\overline{B}_0(\sqrt{\psi_i - |x^1-y_i|^2}).
\]
Since $\rho^\eps$ is concentrated on the
horizontal strip $\R \times (-\eps,\eps)$,
Fubini's theorem yields
\begin{align*}  
\rho^\eps(\RLag_i^{(2)}(\psi))
&= \frac{1}{2\eps} \int_{\Lag_i(\psi)}
2 \min\{\sqrt{\psi_i - |x^1-y_i|^2},\eps\}
\rho(x^1) dx^1 \\
&= \int_{\Lag_i(\psi)} \min\left\{
\eps^{-2}(\psi_i - |x^1-y_i|^2),1\right\}
d\rho(x^1) \\
&= \int_{\Lag_i(\psi)}
(f_\eps^*)'(\psi_i-|x^1-y_i|^2) d\rho(x^1),
\end{align*}
and this is precisely the regularized mass
of the one-dimensional Laguerre cell.
Figure~\ref{fig:2D-diagram} illustrates
this two-dimensional interpretation of
the regularized problem.

\begin{figure}
    \centering
    \includegraphics[width=\textwidth]{
    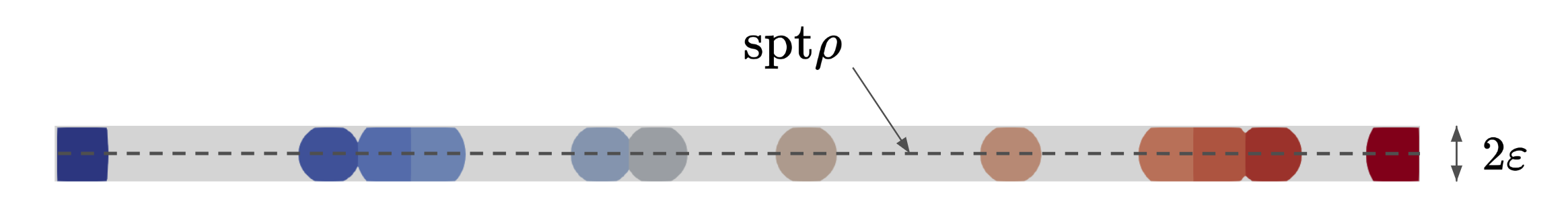}
    \caption{Visualization of the
    two-dimensional problem
    corresponding to the regularization.
    The long gray rectangle is the support
    of $\rho^\eps$ and the colored
    areas are the two-dimensional
    restricted Laguerre cells.}
    \label{fig:2D-diagram}
\end{figure}

\begin{remark}
\label{rem:cells-in-increasing-order}
Since the Laguerre of the
two-dimensional problem are of the form
$\Lag^{(2)}_i(\psi)
= \Lag_i(\psi)\times\R$,
they form a tessellation the plan into
$N$ vertical strips, whose corresponding
indices are increasing from left to right.
\end{remark}

We can, then, establish the following convergence result.

\begin{proposition}
\label{prop:gamma-convergence-regularized-SDOPT}
The functional
$-\mathcal{K}^\eps$
$\Gamma$-converges to $-\mathcal{K}$
as $\eps\to 0$.
\end{proposition}
\begin{proof}
Suppose that
$\eps_n\xrightarrow[n\to\infty]{}0$
and $\psi^n\xrightarrow[n\to\infty]{}\psi$,
and denote $f_n := f_{\eps_n}$.
By dominated convergence, we have
\[
\int_{\R}
\mathbbm 1_{\Lag_i(\psi^n)}(x)
(f_n^*)'(\psi^n_i-|x-y_i|^2)\rho(x)dx
\xrightarrow[n\to\infty]{}
\int_{\R}
\mathbbm 1_{\Lag_i(\psi)}(x)
\mathbbm 1_{\{|x-y_i|^2\leq\psi_i\}}(x)\rho(x)dx.
\]
Indeed, the integrand is dominated
by the integrable density $\rho$,
the set $\{x\in\R :
|x-y_i|^2=\psi_i\}$
is negligible, and for
$x$ such that
$|x-y_i|^2\neq\psi_i$ we have
\begin{align*}
(f_n^*)'(\psi^n_i-|x-y_i|^2)
&= (f^*)'(\eps_n^{-2}(\psi^n_i-|x-y_i|^2)) \\
&\longrightarrow
\begin{cases}
\lim\limits_{+\infty} f' = 1
\quad \mathrm{if} \quad
|x-y_i|^2<\psi_i, \\
\lim\limits_{-\infty} f' = 0
\quad \mathrm{if} \quad
|x-y_i|^2>\psi_i.
\end{cases}
\qedhere
\end{align*}
\end{proof}

\section{An ODE
characterization of the regularized problem}

The maximizer $\psi^\eps$ of $\mathcal{K}^\eps$
is characterized by
$\nabla \mathcal K^\eps (\psi^\eps) = 0$,
that is, $G^\eps(\psi^\eps) = \alpha$,
where we recall that $G^\eps(\psi)$
is the vector of regularized masses
of the Laguerre cells corresponding to $\psi$.
Thanks to these optimal conditions,
the implicit function theorem yields an ODE
satisfied by $\eps \mapsto \psi^\eps$ which,
combined with some estimates on the derivatives
of $G^\eps$, allows to quantify the convergence
of $\psi^\eps$ to the solution $\psi^*$
of the unregularized problem.
For every $\eps>0$,
the uniqueness of $\psi^\eps$ is guaranteed
by Theorem~\ref{thm:regularity-multiD}
and Assumption~\ref{a:density-SDOPT-1D}.


\begin{theorem}
\label{thm:ODE-and-convergence}
Suppose that $\rho$ satisfies
Assumption~\ref{a:density-SDOPT-1D}.
Then there exists $\eps_0>0$ such that
for any $\eps\in(0,\eps_0)$,
the regularized functional $\mathcal{K}^\eps$
has a unique maximizer $\psi^\eps\in\R^N$.
The function $\eps \mapsto \psi^\eps$ is of class
$\Cc^1$ on $(0,\eps_0)$,
and on this interval we have
\begin{equation}
\label{eq:ODE}
[D_{\psi} G^\eps(\psi^\eps)]\dot\psi^\eps
+ (\partial_\eps G^\eps)(\psi^\eps) = 0,
\end{equation}
where $D G^\eps =
-D^2 \mathcal K^\eps$.
Yet again on this interval, it holds
\begin{equation}
\label{ineq:quadratic-rate}
\|\psi^\eps-\psi^*\|
\lesssim \eps^2,
\end{equation}
where $\psi^*\in(0,+\infty)^N$ is the maximizer
of the unregularized functional $\mathcal{K}$
and where $\lesssim$ hides a constant
depending only on $\rho$, on
the $y_i$ and on the $\alpha_i$.
\end{theorem}

Moreover, this behavior is widely observed
in numerical tests, as illustrated
in Figure~\ref{fig:regularization-error}.
The following example shows that
the quadratic convergence rate obtained
in this section is tight.

\begin{example}
\label{ex:quadratic-rate-is-tight}
To see that the quadratic rate given in
Theorem~\ref{thm:ODE-and-convergence}
is tight, consider the simple case
where $\rho$ is the uniform probability
measure on $[0,1]$ and where the discrete
measure is $\mu = \frac{1}{2}\delta_0$.
The optimal $\psi\in\R^1$ for the unregularized
problem is $\psi^*=\frac{1}{4}$, and
the optimal regularized mass of the
Laguerre cell for $\eps>0$ writes
\begin{align*}
\textstyle
\frac{1}{2}
=G^\eps(\psi^\eps)
=\frac{1}{2}\sqrt{\psi^\eps-\eps^2}
+ \frac{\psi^\eps}{2\eps}\arcsin
\left(\frac{\eps}{\sqrt{\psi^\eps}}
\right)
= \sqrt{\psi^\eps}
-\frac{\eps^2}{6\sqrt{\psi^\eps}} + o(\eps^2)
\end{align*}
asymptotically as $\eps\to 0^+$,
since $\sqrt{\psi^\eps}\to\frac{1}{2}$.
We thus have
$\sqrt{\psi^\eps}-\frac{1}{2}
\sim \frac{\eps^2}{3}$
and so
\begin{equation*}
\textstyle
\psi^\eps-\psi^*
= \left(\sqrt{\psi^\eps}+\frac{1}{2}\right)
\left(\sqrt{\psi^\eps}-\frac{1}{2}\right)
\sim \frac{\eps^2}{3}.
\end{equation*}
\end{example}

\subsection{Notations}
In this subsection, we exploit the 2D
interpretation of the regularization
introduced above, so the notations
$\Lag_i$ and $\RLag_i$
refer respectively to the \textit{two-dimensional}
unrestricted and restricted Laguerre cells.
In consideration of
Remark~\ref{rem:auxiliary-point},
the complement of
$\cup_{i=1}^N\RLag_i$
will be denoted $\RLag_\infty$.
To further lighten the notation,
we also introduce the following definitions
for the quantities related to the
optimal potential $\psi^\eps$.
For $i\in\Iint{1,N}\cup\{\infty\}$,
$\RLag^\eps_i$ will denote
the intersection of
$\RLag_i(\psi^\eps)$
with the rectangle
$\spt\rho_\eps
= \spt\rho
\times[-\eps,\eps]$, and for $i\neq\infty$,
$B_i^\eps$ will be the open disk
of center $(y_i,0)$ and of radius
$\sqrt{\psi_i^\eps}$.
The Hessian matrix
$-D G^\eps(\psi^\eps) =
-D^2\mathcal{K}^\eps(\psi^\eps)$
and the mixed derivative
$-(\partial_\eps G^\eps)(\psi^\eps) =
-\partial_\eps(\nabla \mathcal{K}^\eps)(\psi^\eps)$
will respectively be denoted $H^\eps$ and $q^\eps$,
so that equation~\eqref{eq:ODE} writes
$H^\eps\dot\psi^\eps = q^\eps$.
Note that there is a minus sign in front
of both quantities, so that in particular
$H^\eps$ is \textit{positive} semi-definite.
At last, we denote
$\alpha_{\min}=\min\{\alpha_1,\dots,\alpha_N\}$
and $\alpha_\infty=1-\sum_{i=1}^N\alpha_i$.

\subsection{Proof of the theorem}

In order to prove
Theorem~\ref{thm:ODE-and-convergence},
we first need to control
the minimum eigenvalue
of $H^\eps$ --- which is
a positive semi-definite matrix since
$\mathcal{K}^\eps$ is concave ---
independently of $\eps$,
which is the main difficulty.
The argument is based on results of
spectral graph theory, and we thus introduce
the concept of Laplacian matrices.
Only simple (i.e.~without loops),
\textit{undirected} graphs will be
considered throughout the paper.

\begin{definition}
A \textnormal{Laplacian matrix} is any symmetric
matrix $M=(M_{i,j})_{0\leq i,j\leq N}$
with non-negative off-diagonal entries
such that each line sums to zero. It corresponds
to the unique weighted graph with vertices
$\{1,\dots,N,\infty\}$
and weights $w_{i,j} = |M_{i,j}|$,
where for convenience we identify the indices
$0$ and $\infty$.
Two vertices $i$ and $j$ are called
\textnormal{neighbors} if $|M_{i,j}| > 0$,
and we say the matrix $M$ \textnormal{connected}
if the corresponding graph is connected.
\end{definition}

In this work, the indices of a square matrix
with $N+1$ rows will always range from $0$ to $N$,
whereas they will range from $1$ to $N$ for
a square matrix with $N$ rows.
As will be apparent in the the definition
of equation~\eqref{eq:M^eps-matrix},
in our case the index zero corresponds to the
fictitious point that was the subject of
Remark~\ref{rem:auxiliary-point}.
We therefore sometimes use the symbol $\infty$
for this index.
The well known properties given in the
following lemma can be found in Spielman's
book~\cite[Chapter~3]{spielman2019spectral}.

\begin{lemma}
\label{lem:Laplacian-matrices}
Let $M\in\R^{(N+1)\times(N+1)}$
be a Laplacian matrix.
Then $M$ is symmetric positive semi-definite
and its smallest eigenvalue is $0$,
with $\mathbbm 1$ (the vector of ones)
an associated eigenvector.
We thus denote
$0=\lambda_0(M)\leq \lambda_1(M)
\leq \dots \leq \lambda_N(M)$ its
eigenvalues with multiplicities.
The second value is given by
\begin{equation}
\label{eq:1st-eigenvalue-of-Laplacian-matrix}
\lambda_1(M) =
\min_{\substack{\|v\|=1\\v\perp\mathbbm 1}}v^T M v,
\end{equation}
and is non-zero if and only if the weighted graph
associated to $M$ is connected.
\end{lemma}

Because of the correspondence
mentioned above between Laplacian matrices and
weighted graphs on $\{1,\dots,N,\infty\}$,
for any such graph $\mathcal{G}$ we denote
$0=\lambda_0(\mathcal{G})\leq
\lambda_1(\mathcal{G})
\leq \dots \leq \lambda_N(\mathcal{G})$
the eigenvalues with multiplicities
of its Laplacian matrix.
The following lemma states that
the second smallest eigenvalue is
non-decreasing in the weights of the graph.

\begin{lemma}
\label{lem:Laplacian-matrix-dominance}
A Laplacian matrix $M$ is said to dominate
another Laplacian matrix $M'$ if
$|M_{i,j}| \geq |M_{i,j}'|$ for
all distinct indices $i \neq j$.
If this is the case,
then $\lambda_1(M) \geq \lambda_1(M')$.
\end{lemma}
\begin{proof}
Suppose $M$ dominates $M'$.
Then $M-M'$ is semi-definite positive,
so writing $M=M'+(M-M')$ and distributing
the minimum in~\eqref{eq:1st-eigenvalue-of-Laplacian-matrix}
yields the desired inequality.
\end{proof}

Now, let $M^\eps\in\R^{(N+1)\times(n+1)}$ be the
tridiagonal, symmetric matrix defined by
\begin{align}
\label{eq:M^eps-matrix}
\begin{cases}
\displaystyle
M_{i,i+1}^\eps =
\frac{-1}{2(y_{i+1}-y_i)}
\int_{\RLag_i^\eps
\cap \RLag_{i+1}^\eps}
\rho^\eps(x) d\H^1(x)
\quad &\mathrm{for} \quad i\in\Iint{1,N-1}, \\
\displaystyle
M_{i,0}^\eps =
\frac{-1}{2\sqrt{\psi_i^\eps}}
\int_{\RLag_i^\eps
\cap \partial B_i^\eps}
\rho^\eps(x) d\H^1(x)
\quad &\mathrm{for} \quad i\in\Iint{1,N}, \\
M_{i,i}^\eps = - \sum_{j\neq i} M_{i,j}^\eps
\quad &\mathrm{for} \quad i\in\Iint{0,N},
\end{cases}
\end{align}
with all other coefficients null, and call
$\mathcal{G}^\eps$ the corresponding weighted
graph on $\{1,\dots,N,\infty\}$.
Indeed, it is straightforward to check that
$M^\eps$ is a Laplacian matrix.
Note that by Theorem~\ref{thm:regularity-multiD}
and Proposition~\ref{prop:functional-reg-SDOPT-1D},
the submatrix obtained by removing the line and
the column of index $0$ --- which correspond
the auxiliary point --- is precisely
the tridiagonal matrix $H^\eps$,
the minimal eigenvalue of which we seek
to bound from below.
Thanks to the following lemma, it suffices to
find a lower bound for $\lambda_1(M^\eps)$.

\begin{lemma}
\label{lem:min-eigenvalue-of-submatrix}
Let $M=(M_{i,j})_{0\leq i,j\leq N}$ be a
Laplacian matrix, and let $\tilde M$
be the submatrix obtained by removing the
line and the column of index $0$. Then
$\lambda_{\min}(\tilde M)$,
the smallest eigenvalue of $\tilde M$,
satisfies
\begin{equation*}
\lambda_{\min}(\tilde M)
\geq \frac{\lambda_1(M)}{N+1}.
\end{equation*}
\end{lemma}
\begin{proof}
Denote $P$ the
orthogonal projection from $\R^{N+1}$ onto
$\mathbbm 1^\perp\subseteq \R^{N+1}$,
let $u\in\R^N$ and define
$\underline{u} = (0,u_1,\dots,u_N)\in\R^{N+1}$.
Since $M \mathbbm 1 = 0$, we have
\begin{equation}
\label{ineq:lower-bd-for-Rayleigh-M-tilde}
u^T \tilde M u
= \underline{u}^T M \underline{u}
= (P\underline{u})^T M (P\underline{u})
\geq \lambda_1(M) \|P\underline{u}\|^2.
\end{equation}
Moreover, using the additional notation
$e=(N+1)^{-1/2}\mathbbm 1\in\mathbb S^{N+1}$,
we find
\begin{equation*}
\|P\underline{u}\|^2
= \|\underline{u}-(e^T\underline{u})e\|^2
= \|\underline{u}\|^2 - (e^T\underline{u})^2
\end{equation*}
and $(e^T\underline{u})^2 =
(N+1)^{-1}(\sum_{i=1}^Nu_i)^2
\leq N(N+1)^{-1}\|u\|^2$
by Cauchy-Schwarz inequality,
so $\|P\underline{u}\|^2 \geq (N+1)^{-1}\|u\|^2$.
Inequality~\eqref{ineq:lower-bd-for-Rayleigh-M-tilde}
and the characterization
of the smallest eigenvalue of $\tilde M$
as the minimum of the Rayleigh quotient
$\frac{u^T\tilde Mu}{\|u\|^2}$ over all non-zero
vectors $u$ yields the desired inequality.
\end{proof}

Our strategy for getting a lower bound on 
$\lambda_1(M^\eps)$ which is independent of 
$\eps$ is to compare the Laplacian matrix 
$M^\eps$ to that of a uniformly weighted graph.
In order to do so,
we need to control from below the absolute values
of the non-zero off-diagonal entries $|M^\eps_{i,i+1}|$
and $|M^\eps_{i,0}|$.
This in turn requires us to control the components
of $\psi^\eps$.
The object of the following lemma is to show that
for $\eps$ small enough,
these components are uniformly
bounded away from zero and infinity.
Roughly speaking, the radii $\sqrt{\psi^\eps_i}$
of the balls $B_i^\eps$ cannot be two small
because the cells have mass bounded below and
because the density is upper bounded ;
they cannot be too large either because
the auxiliary cell is not empty.

\begin{lemma}
\label{lem:psi^eps-bounded}
If $0<\eps<1$,
then for all $i\in\Iint{1,N}$ we have
\[
r \leq
\sqrt{\psi_i^\eps} \leq R,
\]
where
$r= \frac{\alpha_{\min}}{2\rho_{\max}}$
and
$R = \left(1 + \diam\left(
\spt\rho\cup\{y_1,\dots,y_N\}\right)^2
\right)^{\frac{1}{2}}$.
Moreover, for any \mbox{$0 < \eps < r$}, the
cells $\RLag_i^\eps$ all intersect the
horizontal lines of \mbox{ordinate $\pm\eps$.}
In consideration of these two properties,
we define $\eps_0=\min(1,r)$.
\end{lemma}
\begin{proof}
We first prove the lower bound on
$\sqrt{\psi_i^\eps}$. To do so,
we note that $\RLag_i^\eps$
is included in a rectangle of width $2\eps$
and of height $2\sqrt{\psi_i^\eps}$.
Since $G_i^\eps(\psi^\eps) = \alpha_i$
and since $\rho^\eps$ is bounded above
by $\frac{\rho_{\max}}{2\eps}$, we get
$\alpha_i\leq\frac{\rho_{\max}}{2\eps}
\times 2\eps\times2\sqrt{\psi_i^\eps}$,
which yields the lower bound.
We now derive the upper bound.
The cell $\RLag_\infty^\eps$
associated to the auxiliary point has non-zero
$\rho^\eps$-measure $\alpha_{\infty} > 0$,
so it must intersect the support of $\rho^\eps$.
Fix $(x,w)$ a point in this intersection.
It satisfies
$\psi^\eps_i \leq |x-y_i|^2+|w-0|^2$
for all $i\in\Iint{1,N}$.
The first square in the left-hand side is bounded by
by the squared diameter of
$\spt\rho\cup\{y_1,\dots,y_N\}$,
while the second one is bounded by $\eps^2$,
and the upper bound follows.

Let us now show that for $\eps$ small enough,
the optimal cells all intersect the horizontal
lines $\R\times\{\pm\eps\}$.
Fix $\eps\in(0,r)$ and suppose by contradiction
that, for some $i\in\Iint{1,N}$, the cell
$\RLag_i^\eps$ does not
intersect these horizontal lines.
Then the $\rho^\eps$-measure of the cell is
bounded above by $\frac{\rho_{\max}}{2\eps}$
times the area of the two disk segments shown
in Figure~\ref{fig:disk-segments}
(taking $\omega=\sqrt{\psi^\eps_i}$).
Recalling that $\sqrt{\psi^\eps_i}\geq r$,
each of these disk segments has area
bounded above by $\frac{\pi\eps^3}{2r}$,
so we have
$\alpha_{\min} \leq \rho^\eps(\RLag^\eps_i)
\leq 2\times\frac{\rho_{\max}}{2\eps}\times
\frac{\pi\eps^3}{2r} =
\frac{\pi\rho_{\max}}{2r}\eps^2$.
As a result, $\eps\geq
\sqrt{\frac{2r\alpha_{\min}}{\pi\rho_{\max}}}
= \sqrt{\frac{4}{\pi}}r>r$,
which is a contradiction.
Note that since we only assumed $\eps < r$,
the rougher bound mentioned in
Figure~\ref{fig:disk-segments} actually
suffices to get the wanted contradiction.
\end{proof}

\begin{figure}
\centering
\begin{tikzpicture}[scale = 2.]
    \draw[black, opacity = 0.8] (0., -1.) -- (6., -1.) -- (6., 1.) -- (0, 1.) -- cycle;
    \draw[dotted,black,opacity=1,
    line width=1pt] (1,1) -- (1,-1);
    \draw[dotted,black,opacity=1,
    line width=1pt] (5,1) -- (5,-1);
    \draw[dashed,black,opacity=0.5] (0,0) -- (6,0);
    \fill (3,0) circle[radius=0.7pt];

    \draw[decorate,decoration={brace,raise=5pt,amplitude=5pt,mirror}]
    (0,1) -- node[left=10pt]{$\varepsilon$} (0,0) ;
    \draw[decorate,decoration={brace,raise=5pt,amplitude=5pt}]
    ({3-sqrt(3)},1) -- node[]{} (3,0) ;
    \draw (2.3,0.8) node {$\omega$} ;

    \draw[decorate,decoration={brace,raise=5pt,amplitude=5pt,mirror}]
    ({3-sqrt(3)},1) --
    node[above=10pt]{$\omega-\sqrt{\omega^2-\varepsilon^2}$} (1,1) ;
    \draw[decorate,decoration={brace,raise=5pt,amplitude=5pt,mirror}]
    ({3-sqrt(3)},-1) --
    node[below=10pt]{$\sqrt{\omega^2-\varepsilon^2}$} (3,-1) ;

    \clip (3,0) circle (2);
    \begin{scope}
    \fill[blue, opacity=0.4] ({3-sqrt(3)},-1) rectangle (1,1);
    \draw[black, opacity=1.] ({3-sqrt(3)},-1) rectangle (1,1);
    \end{scope}
    \draw[black, line width=0.8, opacity=1.] (3,0) ++(150:2) arc (150:210:2);
    
    \clip (3,0) circle (2);
    \begin{scope}
    \filldraw[blue, opacity=0.4] ({3+sqrt(3)},-1) rectangle (5,1);
    \draw[black, opacity=1.] ({3+sqrt(3)},-1) rectangle (5,1);
    \end{scope}
    \draw[black, line width=0.8, opacity=1.] (3,0) ++(30:2) arc (30:-30:2);
\end{tikzpicture}
\vspace*{-1.5cm}
    \caption{The area of each of the two
    disk segments in blue is
    $\omega^2(\arcsin(\eps/\omega)
    -\eps/\omega\sqrt{1-(\eps/\omega)^2})$.
    A simple computation shows that for $u\in[0,1]$
    we have $\arcsin(u)-u\sqrt{1-u^2}\leq\pi u^3/2$,
    so that the latter area is at most
    $\pi\eps^3/(2\omega)$.
    A rougher estimate consists in applying
    $1-u^2\leq\sqrt{1-u^2}$ to $u=\eps/\omega$,
    which yields the upper bound
    $2\eps^3/\omega$ for the area.}
    \label{fig:disk-segments}
\end{figure}
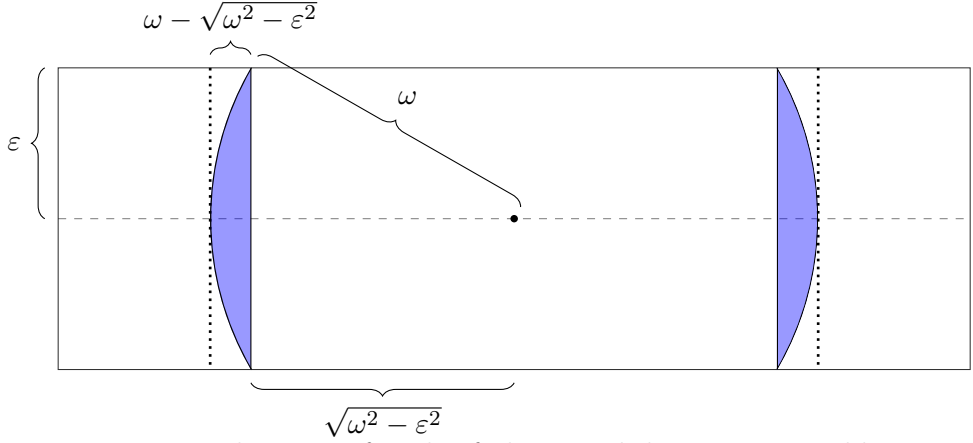


\begin{lemma}
\label{lem:M-eps-uniformly-connected}
If $0<\eps<\eps_0$,
then for all $i \in \Iint{1,N-1}$
we have
\[
|M_{i,i+1}^\eps| \geq \beta
\qquad \mathrm{or} \qquad
\begin{cases}
|M_{i,0}^\eps| \geq \beta, \\
|M_{i+1,0}^\eps| \geq \beta,
\end{cases}
\]
where
$\beta = \frac{\rho_{\min}}{4R}$.
\end{lemma}

From this lemma, we construct as follows a graph
$\mathcal{G}^\eps_{\mathrm{u}}$ on
$\{1,\dots,N,\infty\}$
with weights in $\{0,1\}$.
For any two distinct vertices $i\neq j$,
we put a weight equal to one between $i$ and $j$
if and only if $|M_{i,j}^\eps|\geq\beta$.
All the other weights are zero.
Figure~\ref{fig:graph-G^eps_u} provides
an example of a possible graph constructed
in this way.
Lemma~\ref{lem:M-eps-uniformly-connected},
implies that $\mathcal{G}^\eps_{\mathrm{u}}$
is connected.
The proof amounts to showing that
if two consecutive vertices $i$ and $i+1$
are not connected in the graph, then both are
connected to the auxiliary vertex $\infty$,
cf Figure~\ref{fig:sets-L-E-R}.

\begin{figure}
\centering
\centering
\begin{tikzpicture}[auto, 
node distance = 1 cm and 1cm,
on grid, semithick,
state/.style = {circle,
top color = white, draw,
text=black, minimum width =0.2 cm},
every loop/.style={},]

\node[state] (P1) {};
\node[state] (P2) [right= of P1] {};
\node[state] (P3) [right= of P2] {};
\node[state] (P4) [right= of P3] {};
\node[state] (P5) [right =of P4] {};
\node[state] (P6) [right =of P5] {};
\node[state] (P7) [right =of P6] {};
\node[state] (P8) [right =of P7] {};
\node[state] (P9) [right =of P8] {};
\node[state] (Pinf) [above =of P6] {};
\path[-] (P1) edge (P2);
\path[-] (P2) edge (P3);
\path[-] (P3) edge (P4);
\path[-] (P4) edge (Pinf);
\path[-] (P5) edge (Pinf);
\path[-] (P5) edge (P6);
\path[-] (P6) edge (P7);
\path[-] (P7) edge (Pinf);
\path[-] (P8) edge (Pinf);
\path[-] (P8) edge (P9);

\foreach \i in {1,...,9} {
\draw (\i-1,-0.5)
node[anchor=center] {$\i$};
}

\draw (5,1.5)
node[anchor=center] {$\infty$};
\end{tikzpicture}

\caption{Example of a graph
$\mathcal{G}^\eps_{\mathrm{u}}$,
where the connected components of
$\mathcal{G}^\eps_{\mathrm{u}}\setminus\{\infty\}$
are $\{1,2,3,4\}$,
$\{5,6,7\}$ and
$\{8,9\}$.}
\label{fig:graph-G^eps_u}
\end{figure}
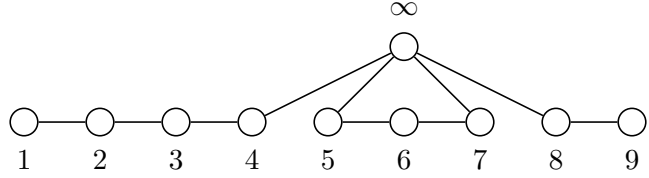

\begin{proof}[Proof of
Lemma~\ref{lem:M-eps-uniformly-connected}]
For $i\in\Iint{1,N}$, define
\begin{equation*}
\begin{cases}
R_i^\eps &=
\RLag_i^\eps\cap\partial B_i^{\eps,+}, \\
L_i^\eps &=
\RLag_i^\eps\cap\partial B_i^{\eps,-},
\end{cases}
\end{equation*}
where
$\partial B_i^{\eps,+} :=
\partial B_i^{\eps}\cap([y_i,+\infty)\times\R)$
and $\partial B_i^{\eps,-} :=
\partial B_i^{\eps}\cap((-\infty,y_i])\times\R)$
are the right-hand and left-hand sides
of the ball $B_i^\eps$.
In addition, let us denote
\begin{align*}
E_i^\eps =
\begin{cases}
\RLag_1^\eps \cap
(\{\min\spt\rho\}\times\R)
\quad&\mathrm{if}\quad i = 0, \\
\RLag_i^\eps \cap
\RLag_{i+1}^\eps
\quad&\mathrm{if}\quad i\in\Iint{1,N-1}, \\
\RLag_N^\eps \cap
(\{\max\spt\rho\}\times\R)
\quad&\mathrm{if}\quad i = N.
\end{cases}
\end{align*}
Since $\eps < r$, Lemma~\ref{lem:psi^eps-bounded}
ensures that the restricted Laguerre cells all
intersects the horizontal lines
$\R\times\{\pm\eps\}$. As a result, for every
$i\in\Iint{1,N-1}$ we have
\begin{equation}
\begin{cases}
\label{ineq:sets-L-E-R}
\H^1(E_i^\eps)
+ \H^1(R_i^\eps)
&\geq 2\eps, \\
\H^1(E_i^\eps)
+ \H^1(L_{i+1}^\eps)
&\geq 2\eps,
\end{cases}
\end{equation}
as illustrated in
Figure~\ref{fig:sets-L-E-R}.
Recall that $\rho^\eps$ is bounded below by
$\frac{\rho_{\min}}{2\eps}$ on its support,
that the components of $\psi^\eps$ are all bounded
above by $R$, and that the distance between
the positions of two consecutive Dirac masses
is at most $R$.
From~\eqref{eq:M^eps-matrix} and from the
definition of $\beta$ we deduce that
\begin{align*}
\begin{cases}
|M_{i,i+1}^\eps|&\geq\beta\H^1(E_i^\eps)/\eps,\\
|M_{i,0}^\eps|&\geq\beta\H^1(R_i^\eps)/\eps,\\
|M_{i+1,0}^\eps|&\geq\beta\H^1(L_{i+1}^\eps)/\eps.
\end{cases}
\end{align*}
These inequalities combined
with~\eqref{ineq:sets-L-E-R} allow
to conclude the proof.
\end{proof}

\begin{figure}
\centering
\begin{tikzpicture}[scale=1.5,
line cap=round,
line join=round,
>=triangle 45]
\draw [line width=.5pt]
(0.2,2)--(3.8,2);
\draw [line width=.5pt]
(0.2,0)--(3.8,0);
\draw [line width=.8pt, dotted]
(2,0)--(2,2);
\draw[dashed,black,opacity=1]
(0.2,1) -- (3.8,1);

\begin{scope}
\begin{pgfinterruptboundingbox}
\clip (0.85,1) circle (1.26);
\filldraw[gray!40, opacity=0.4] (0.2,0)
rectangle (2,2);
\end{pgfinterruptboundingbox}
\end{scope}

\begin{scope}
\begin{pgfinterruptboundingbox}
\clip (3.8,1) circle (1.86);
\filldraw[gray!40, opacity=0.4] (2,0)
rectangle (3.8,2);
\end{pgfinterruptboundingbox}
\end{scope}

\draw[line width=.8pt, red]
(2,1.5)--(2,0.5);

\draw[blue, line width=.8pt] (2,1.5) arc[
start angle=203.5,
delta angle=29,
radius=-1.26];

\draw[blue, line width=.8pt] (2,0.5) arc[
start angle=156.5,
delta angle=-29,
radius=-1.26];

\draw[black!30!green, line width=.8pt] (2,1.5) arc[
start angle=164,
delta angle=-17.5,
radius=1.8];

\draw[black!30!green, line width=.8pt] (2,0.5) arc[
start angle=16,
delta angle=17.5,
radius=-1.8];

\draw (0.7, 1.2) node[right] {
$\mathrm{RLag}_i^\varepsilon$};
\draw (2.6, 1.2) node[right] {
$\mathrm{RLag}_{i+1}^\varepsilon$};
\draw[decorate,decoration={brace,raise=5pt,amplitude=5pt,mirror}]
    (0.2,2) -- node[left=10pt]{$2\varepsilon$} (0.2,0) ;


\draw (3.9, 1.7) node[right] {
$\textcolor{blue}{
R_{i}^\varepsilon}$};
\draw (3.9, 1) node[right] {
$\textcolor{red}{
E_{i}^\varepsilon}$};
\draw (3.9, 0.3) node[right] {
$\textcolor{black!30!green}{
L_{i+1}^\varepsilon}$};

\end{tikzpicture}
    \caption{Illustration for the sets $R_{i}^\eps$,
    $E_i^\eps$ and $L_{i+1}^\eps$,
    respectively in blue, in red, and in green.
    The dotted vertical line corresponds
    to the boundary
    between the two unrestricted Laguerre cells.}
    \label{fig:sets-L-E-R}
\end{figure}
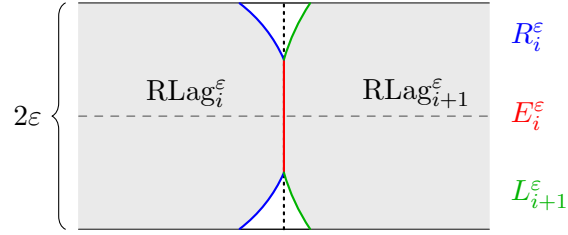

Lemma~\ref{lem:M-eps-uniformly-connected} implies
that for $\eps$ small enough,
$M^\eps$ is uniformly connected.
More precisely, the weights of its
associated graph $\mathcal{G}^\eps$ are bounded
below by the constant $\beta$ times that of the
uniformly weighted graph
$\mathcal{G}^\eps_{\mathrm{u}}$ introduced after
Lemma~\ref{lem:M-eps-uniformly-connected}.
We can thus harness a connectivity inequality
by Fiedler on simple unweighted graphs in order
to get a lower bound for the smallest
non-zero eigenvalue of $M^\eps$. This, in turn,
implies that the matrices $H^\eps$ are
\textit{uniformly} positive definite, thanks to
Lemma~\ref{lem:min-eigenvalue-of-submatrix}.

\begin{remark}
\label{rem:a-priori-G^eps_u-indt-of-eps}
Since the potential $\psi^\eps$ will be
shown to converge to $\psi^*$
(Theorem~\ref{thm:ODE-and-convergence}),
for $\eps$ small enough the graph
$\mathcal{G}^\eps_{\mathrm{u}}$ is in fact
independent of $\eps$. Indeed, two consecutive
vertices $i$ and $i+1$ will be connected if
and only if the corresponding optimal cells
for the unregularized uni-dimensional problem
intersect, and otherwise both will be connected
to $\infty$. Note however that the proof we
provide for Theorem~\ref{thm:ODE-and-convergence}
does not rely on this property.
\end{remark}

\subsubsection*{A graph theory inequality by Fiedler}
The following result can be found
in~\cite[Paragraph~4.3]{fiedler1973algebraic}.
Let $\mathcal{G}$ be graph with $n$ vertices,
and let $e(\mathcal{G})$ denote its
\textit{edge connectivity}, i.e.~the minimum
number of edges whose deletion disconnects the graph.
Then the second smallest
eigenvalue $\lambda_1(\mathcal{G})$
of its Laplacian matrix satisfies
\begin{equation}
\label{eq:Fiedler}
\lambda_1(\mathcal{G}) \geq
2 e(\mathcal{G}) \left(
1 - \cos\left(\pi n^{-1}
\right)\right).
\end{equation}


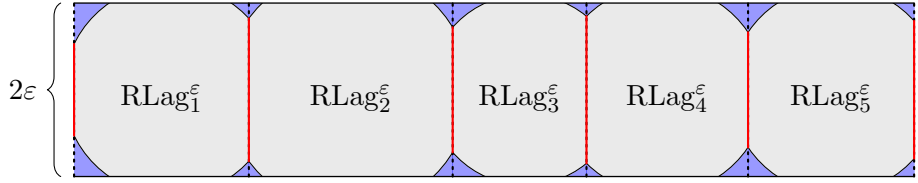
\begin{figure}
\centering
\begin{tikzpicture}[scale=1.15,
line cap=round,
line join=round,
>=triangle 45]

\tikzmath{
\r1 = 1.22; \a1 = 1.1; \b1 = 0.9;
\r2 = 1.4; \a2 = (\r2^2-\r1^2+\b1^2)^0.5; \b2 = 1.2;
\r3 = 1.1; \a3 = (\r3^2-\r2^2+\b2^2)^0.5; \b3 = 0.7;
\r4 = 1.2; \a4 = (\r4^2-\r3^2+\b3^2)^0.5; \b4 = 1;
\r5 = 1.2; \a5 = (\r5^2-\r4^2+\b4^2)^0.5; \b5 = 0.9;}

\def\elts{
{\r1,\a1,\b1},
{\r2,\a2,\b2},
{\r3,\a3,\b3},
{\r4,\a4,\b4},
{\r5,\a5,\b5}
}

\draw[decorate,decoration={brace,raise=5pt,amplitude=5pt,mirror}]
    (0,2) -- node[left=10pt]{$2\varepsilon$} (0,0) ;

\pgfmathsetmacro{\cumul}{0};
\pgfmathsetmacro{\index}{1};
\foreach \e in \elts {
    \begin{scope}[shift={(\cumul,0)}]
        
    \begin{scope}
    \clip (0,0) rectangle ({\e}[1]+{\e}[2],2);
    \filldraw[blue, opacity=0.4] (0,0)
    rectangle ({\e}[1]+{\e}[2],2);
    \end{scope}
    
    \begin{scope}
    \clip (0,0) rectangle ({\e}[1]+{\e}[2],2);
    \filldraw[white] ({\e}[1],1) circle ({\e}[0]);
    \filldraw[gray!40, opacity=0.4] ({\e}[1],1) circle ({\e}[0]);
    \end{scope}
    
    \begin{scope}
    \clip (0,0) rectangle ({\e}[1]+{\e}[2],2);
    \draw ({\e}[1],1) circle ({\e}[0]);
    \end{scope}
    
    \draw [line width=.5pt] (0,2)--({\e}[1]+{\e}[2],2);
    \draw [line width=.5pt] (0,0)--({\e}[1]+{\e}[2],0);
    \draw [line width=.8pt, dotted] (0,0)--(0,2);
    \draw [line width=.8pt, red]
    (0,{1-({\e}[0]^2-{\e}[1]^2)^0.5}) --
    (0,{1+({\e}[0]^2-{\e}[1]^2)^0.5});
    \draw ({\e}[1]/2+{\e}[2]/2, 0.9)
    node[anchor=center] {
    $\mathrm{RLag}_{\index}^\varepsilon$};

    \end{scope}
    \pgfmathsetmacro{\ff}{\cumul+{\e}[1]+{\e}[2]};
    \global\let\cumul=\ff
    \pgfmathsetmacro{\ff}{int(\index+1)};
    \global\let\index=\ff
}

\begin{scope}[shift={(\cumul,0)}]
\draw [line width=.8pt, dotted] (0,0)--(0,2);
\draw [line width=.8pt, red]
(0,{1-(\r5^2-\b5^2)^0.5}) --
(0,{1+(\r5^2-\b5^2)^0.5});
\end{scope}
\end{tikzpicture}
    \caption{The restricted Laguerre
    cells represented
    on the support of $\rho^\eps$. The
    vertical segments in red correspond to the sets
    $E_i^\eps$, while the blue area is the
    relative complement of the union of the Laguerre
    cells in $\spt\rho^\eps$.}
    \label{fig:RLag_infty-2D}
\end{figure}

\begin{figure}
\centering
\begin{subfigure}{.45\textwidth}
    \centering
    \begin{tikzpicture}[scale=1.5,
line cap=round,
line join=round,
>=triangle 45]

\begin{scope}
\clip (0.95,0) rectangle (2.95,2);
\filldraw[blue, opacity=0.4] (0.95,0)
rectangle (2.95,2);
\end{scope}

\begin{scope}
\clip (0.95,0) rectangle (2.95,2);
\filldraw[white] (2,1) circle (1.22);
\filldraw[gray!40, opacity=0.4] (2,1) circle (1.22);
\end{scope}

\begin{scope}
\clip (0.95,0) rectangle (2.95,2);
\draw (2,1) circle (1.22);
\end{scope}

\draw [line width=.8pt, dotted]
(0.95,0)--(0.95,2);
\draw [line width=.8pt, dotted]
(2.95,0)--(2.95,2);
\draw [line width=.5pt] (0.2,2)--(3.8,2);
\draw [line width=.5pt] (0.2,0)--(3.8,0);

\draw (1.65, 1.5) node[right] {
$\mathrm{RLag}_i^\varepsilon$
};
\draw[decorate,decoration={brace,raise=5pt,amplitude=5pt,mirror}]
    (0.2,2) -- node[left=10pt]{$2\varepsilon$} (0.2,0) ;
\draw[decorate,decoration={brace,raise=3pt,amplitude=5pt,mirror,
aspect=0.25}]
    (0.95,1.62) -- node[left=10pt,pos=0.25]{$\geq\varepsilon$}
    (0.95,0.38) ;
\draw[decorate,decoration={brace,raise=3pt,amplitude=5pt,
aspect=0.25}]
    (2.95,1.76) -- node[right=10pt,pos=0.25]{$\geq\varepsilon$}
    (2.95,0.24) ;
\fill (2,1) circle[radius=1pt];
\draw[decorate,
decoration={brace,raise=0pt,
amplitude=0pt,mirror}]
    ({2-1.22*sqrt(1/2)},{1-1.22*sqrt(1/2)})
    -- node[right=10pt]{} (2,1) ;
\draw (1.6, 0.4) node[right] {
$\text{radius}\geq\varepsilon$};
\draw [line width=.8pt, dashed,
opacity=0.5] (0.2,1)--(3.8,1);

\end{tikzpicture}    
    \caption{A cell contained in the corresponding
    rectangle of the partition.}
    \label{fig:area-of-Lag-minus-RLag}
\end{subfigure}
\hfill
\begin{subfigure}{.45\textwidth}
    \centering
    \begin{tikzpicture}[scale=1.5,
line cap=round,
line join=round,
>=triangle 45]
\begin{scope}
\clip ({2-sqrt(3)/2},0) rectangle
({2+sqrt(3)/2},2);
\filldraw[blue, opacity=0.4] (1,0)
rectangle (3,2);
\end{scope}

\begin{scope}
\clip ({2-sqrt(3)/2},0) rectangle
({2+sqrt(3)/2},2);
\filldraw[white] (2,1) circle (1);
\filldraw[gray!40, opacity=0.4]
(2,1) circle (1);
\end{scope}

\begin{scope}
\clip ({2-sqrt(3)/2},0) rectangle ({2+sqrt(3)/2},2);
\draw (2,1) circle (1);
\end{scope}

\draw [line width=.8pt, dotted]
({2-sqrt(3)/2},0)--({2-sqrt(3)/2},2);
\draw [line width=.8pt, dotted]
({2+sqrt(3)/2},0)--({2+sqrt(3)/2},2);
\draw [line width=.5pt] (0.2,2)--(3.8,2);
\draw [line width=.5pt] (0.2,0)--(3.8,0);

\draw[decorate,
decoration={brace,raise=5pt,
amplitude=5pt,mirror}]
    (0.2,2) -- node[left=10pt]{
    $2\varepsilon$} (0.2,0) ;
\draw[decorate,
decoration={brace,raise=3pt,
amplitude=5pt,
aspect=0.25,mirror}]
    ({2-sqrt(3)/2},1.5) -- node[
    left=10pt,pos=0.25]{
    $\varepsilon$}
    ({2-sqrt(3)/2},0.5) ;
\draw[decorate,
decoration={brace,raise=3pt,
amplitude=5pt,
aspect=0.25}]
    ({2+sqrt(3)/2},1.5) -- node[
    right=10pt,pos=0.25]{
    $\varepsilon$}
    ({2+sqrt(3)/2},0.5) ;
\fill (2,1) circle[radius=1pt];
\draw[decorate,
decoration={brace,raise=0pt,
amplitude=0pt,mirror}]
    ({2-sqrt(1/2)},{1-sqrt(1/2)})
    -- node[right=10pt]{} (2,1) ;
\draw (1.5, 0.4) node[right] {
$\text{radius}=\varepsilon$};
\draw [line width=.8pt, dashed,
opacity=0.5] (0.2,1)--(3.8,1);
\draw[decorate,
decoration={brace,raise=3pt,
amplitude=5pt}]
    ({2-sqrt(3)/2},1.5) -- node[
    left=10pt]{
    $\frac{\varepsilon}{2}$}
    ({2-sqrt(3)/2},2) ;

\draw[decorate,
decoration={brace,raise=3pt,
amplitude=5pt}]
    ({2-sqrt(3)/2},2)
    -- node[above=10pt]{
    $\frac{3\varepsilon}{\sqrt{2}}$} (2,2) ;
\end{tikzpicture}    
    \caption{The blue area is maximal when the disk
    is tangent to the support of $\rho^\eps$.}
    \label{fig:max-area-of-Lag-minus-RLag}
\end{subfigure}
\caption{The blue area in the left
drawing is at most equal to the blue area
in the right drawing.}
\label{fig:are-of-rel-complement-of-RLag}
\end{figure}
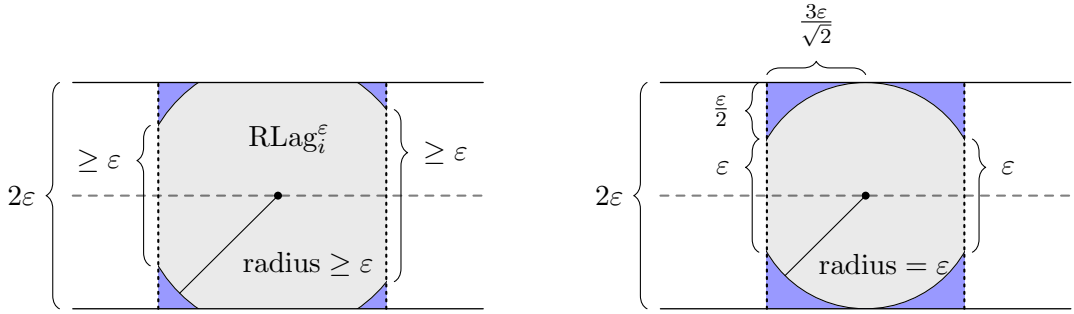

\begin{proposition}
\label{prop:lower-bound-for-H^eps}
For $0<\eps<\eps_0$ we have
\begin{align} 
\label{ineq:lower-bound-for-H^eps}
\lambda_{\min}(H^\eps)
\geq \frac{4\beta}{N+1}
\sin^2\left(\frac{\pi}{2N+2}\right)
\end{align}
where we recall that
\begin{equation*}
4\beta = \rho_{\min}\left(1 + \diam\left(
\spt\rho\cup\{y_1,\dots,y_N\}\right)^2
\right)^{-\frac{1}{2}}.
\end{equation*}
\end{proposition}
\begin{proof}
By Lemma~\ref{lem:M-eps-uniformly-connected}
and Lemma~\ref{lem:Laplacian-matrix-dominance}
we have $\lambda_1(M^\eps) \geq
\beta\lambda_1(\mathcal{G}^\eps_{\mathrm{u}})$.
Since $\mathcal{G}^\eps_{\mathrm{u}}$ is
connected, its edge connectivity is at least $1$,
so Fiedler's connectivity inequality combined with the
identity $1-\cos(2\theta) = 2\sin^2(\theta)$ yields
$\lambda_1(\mathcal{G}^\eps_{\mathrm{u}})\geq
4\sin^2\left(\frac{\pi}{2N+2}\right)$.
Lemma~\ref{lem:min-eigenvalue-of-submatrix}
allows to conclude the proof.
\end{proof}

We now express the derivative of $G_\eps$
with respect to $\eps$, and give
an estimate for its value at the optimal
potential $\psi^\eps$.
The main theorem will then follow
from the implicit function theorem, combined with
the uniform bound~\ref{ineq:lower-bound-for-H^eps}
and the integration of $\eps\mapsto\dot\psi^\eps$
from $0$ to $\eps$.

\begin{proposition}
\label{prop:mixed-derivative-of-G^eps}
Suppose that $0<\eps<\eps_0$.
Then, for every $i\in\Iint{1,N}$,
the partial derivative of $G_i^\eps$
with respect to $\eps$ writes
\begin{align}
\label{eq:mixed-derivative-of-G^eps}
(\partial_\eps G^\eps_i)(\psi)
&=  -\frac{1}{\eps}
\int_{\RLag_i(\psi)} \rho^\eps(x)dx
+ \int_{\RLag_i(\psi)
\cap (\R \times \{-\eps,\eps\})}
\rho^\eps(x)d\H^1(x).
\end{align}
In particular, the function
$(\eps,\psi)\mapsto (\partial_\eps G^\eps)(\psi)$
is jointly continuous on the open set $\Omega$
of pairs $(\eps,\psi)\in(0,\eps_0)\times\R^N$
such that all components of $G^\eps(\psi)$ are non-zero,
as well as $1-\sum_{i=1}^NG^\eps_i(\psi)$.
Moreover, for every $0 < \eps < \eps_0$
we have
\begin{equation}
\label{ineq:bound-for-q^eps}
\left\|q^\eps\right\|\leq
\frac{4\rho_{\max}^2}{\alpha_{\min}}
N^{\frac{1}{2}} \eps.
\end{equation}
\end{proposition}
\begin{proof}
The formula for $(\partial_\eps G^\eps_i)(\psi)$
follows from Fubini's theorem, which allows us
to write
\begin{align*}
G_i^\eps(\psi)
&= \int_{\RLag_i(\psi) \cap
(\R\times[-\eps,\eps])} d\rho^\eps \\
&= \frac{1}{2\eps}
\int_{-\eps}^\eps \left(
\int_{\RLag_i(\psi)
\cap (\R \times\{w\})}
\rho(x)d\H^1(x) \right)dw.
\end{align*}
The joint continuity of
$\partial_\eps G^\eps$ in $(\eps,\psi)$
then immediately follows from the
dominated convergence theorem.
Let us now bound
$q^\eps = -(\partial_\eps G^\eps_i)(\psi^\eps)$.
We take $\eps < \eps_0$, so that every cell
contains a point of ordinate $\eps$.
Leveraging the definition of $\rho^\eps$, we can
write the second term in the right-hand side
of~\eqref{eq:mixed-derivative-of-G^eps}
as $1/\eps$ times the $\rho^\eps$-measure
of the rectangle which is the convex hull of
$\RLag_i^\eps
\cap \partial(\R \times [-\eps,\eps])$.
More precisely,
\begin{align*}
q^\eps_i
&= \frac{1}{\eps} \left[\rho^\eps
(\RLag_i^\eps)
- \rho^\eps(\RLag_i^\eps\cap H_i^\eps)
\right] \\
&= \frac{1}{\eps}\rho^\eps
(\RLag_i^\eps\setminus H_i^\eps),
\end{align*}
where $H_i^\eps$
is the vertical strip of width
$2\sqrt{\psi_i^\eps-\eps^2}$
centered at abscissa $y_i$.
The set $\RLag_i^\eps \setminus
H_i^\eps$ is contained
in the union of two disjoint disk segments
of base $2\eps$ and height $\sqrt{\psi_i^\eps}
-\sqrt{\psi_i^\eps-\eps^2}$, as illustrated in
Figure~\ref{fig:disk-segments}
when choosing $\omega=\sqrt{\psi_i^\eps}$.
Since the height of the disk segments is bounded above
by $\eps^2/\sqrt{\psi_i^\eps}$,
their respective area is at most
$2\eps^3/\sqrt{\psi_i^\eps}$.
The density $\rho^\eps$ is moreover bounded by
$\rho_{\max}/2\eps$, and by
Lemma~\ref{lem:psi^eps-bounded} we have
$\sqrt{\psi_i^\eps}\geq r
=\alpha_{\min}/(2\rho_{\max})$,
so in the end
\begin{equation*}
\left|q^\eps_i\right|
\leq 2\times\frac{\rho_{\max}}{2\eps}\times
\frac{2\eps^3}{\sqrt{\psi_i^\eps}}
= \frac{4\rho_{\max}^2}{\alpha_{\min}}\eps,
\end{equation*}
and the bound on $\|q^\eps\|$ follows.
\end{proof}

\begin{proof}[Proof of
Theorem~\ref{thm:ODE-and-convergence}]
We apply the implicit function theorem
to the function $G : \Omega \to \R^N$ defined by
$G(\eps,\psi)=G^\eps(\psi)$,
where $\Omega$ is the open set defined in
Proposition~\ref{prop:mixed-derivative-of-G^eps}.
Thanks to this proposition and
Theorem~\ref{thm:regularity-multiD},
$G$ has partial first derivatives in both
$\eps$ and $\psi$, and these first derivatives
are jointly continuous on $\Omega$,
so $G$ is $\Cc^1$ on this set.
Note that continuity of
the first derivative in $\psi$ is a direct
consequence of the dominated convergence and
of the fact that the derivative in question
may be written
\begin{equation*}
\partial_{\psi_{i+1}}G_i(\eps,\psi) =
(4(y_{i+1}-y_i)\eps)^{-1}\int_{\R}\rho(z_i(\psi))
\mathbbm 1_{[-w_i(\psi),w_i(\psi)]}(x^2)
\mathbbm 1_{[-\eps,\eps]}(x^2)dx^2,
\end{equation*}
where
$w_i(\psi) = \sqrt{\psi_i-(z_i(\psi)-y_i)^2}$
if the radicand is non-negative,
$w_i(\psi)=0$ otherwise.
Since the Jacobian matrix
$D G^\eps(\psi^\eps)$
is invertible for every $\eps\in(0,\eps_0)$,
the implicit function theorem implies that
$\eps \mapsto \psi^\eps$ is of class
$\Cc^1$ on this interval, with derivative
given by~\eqref{eq:ODE}.

We now prove the convergence of
$\psi^\eps$ to $\psi^*$
as $\eps\to0$.
Let $(\eps_n)$ be a sequence
of strictly positive reals
converging to zero.
By Lemma~\ref{lem:psi^eps-bounded},
the $\psi^{\eps_n}$
are uniformly bounded
so there exists a converging subsequence.
Since the functionals
$-\mathcal{K}^\eps$
$\Gamma$-converge to $-\mathcal{K}$,
the limit of the subsequence
is necessarily $\psi^*$,
the unique maximizer of $\mathcal{K}$.
This holds for any
convergent subsequence of
$(\psi^{\eps_n})$,
so the whole sequence
converges to $\psi^*$.

Now that the convergence of
$\psi^\eps$ to $\psi^*$
is established, let us bound
$\|\psi^\eps-\psi^*\|$.
For $\eps\in(0,\eps_0)$,
equation~\eqref{eq:ODE} and the
inequalities~\eqref{ineq:lower-bound-for-H^eps}
and~\eqref{ineq:bound-for-q^eps} imply that
\begin{align*}
\|\dot\psi^\eps\| &\leq
\frac{1}{\lambda_{\min}(H^\eps)}
\left\|q^\eps\right\| \\
&\leq \frac{N+1}{4\beta
\sin^2\left(\frac{\pi}{2N+2}\right)} \times
\frac{2\rho_{\max}}{r}N^{\frac{1}{2}}\eps.
\end{align*}
As a result,
\begin{align*}
\|\psi^\eps-\psi^*\|
&\leq \int_0^\eps \|\dot\psi^\tau\| d\tau
\leq 2C\int_0^\eps \tau d\tau
= C \eps^2
\end{align*}
where, using the definitions of $\beta$ and $r$,
\begin{equation*}
C = \frac{(N+1)\rho_{\max}N^{\frac{1}{2}}}{
8\beta\sin^2\left(\frac{\pi}{2N+2}\right)r}
= \frac{\rho_{\max}R(N+1)N^{\frac{1}{2}}}{
2\sin^2\left(\frac{\pi}{2N+2}\right)\rho_{\min}r}
= \frac{\rho_{\max}^2R(N+1)N^{\frac{1}{2}}}{
\alpha_{\min} \rho_{\min}
\sin^2\left(\frac{\pi}{2N+2}\right)},
\end{equation*}
and we recall that
$R = \displaystyle\left(1 + \diam\left(
\spt\rho\cup\{y_1,\dots,y_N\}\right)^2
\right)^{\frac{1}{2}}$.
\end{proof}

\section{Numerics}

The aim of this section is to illustrate
the algorithm's behavior, particularly in terms of
accuracy and execution speed.
The solver is based on Newton's algorithm
applied to the regularized dual problem.
To avoid leaving the domain on which
the functional is $\Cc^2$,
backtracking has been implemented:
if the potentials proposed at a given step
lead to empty cells, we go back to the
previous step and we apply a relaxation coefficient
twice smaller. Each step of the algorithm
writes
\begin{equation}
\label{eq:Newton's-Algorithm}
\psi^{(k+1)} = \psi^{(k)} + \eta^{(k)}d^{(k)},
\end{equation}
where $d^{(k)} = -DG^\eps(\psi^{(k)})[
G^\eps(\psi^{(k)})-\alpha]$
is the direction of descent and
$\eta^{(k)}$ is the step size given by the
backtracking line search.

\subsection{Implementation details}
The implementation naturally exploits the
properties of the one-dimensional case.
First, the restricted Laguerre cells in 1D are
straightforward to compute once the
positions of the Dirac masses are sorted,
see equation~\eqref{eq:extremities-of-RLag-1D} in
Appendix~\ref{sec:appendix-uniqueness-SDOPT-1D}.
Second, the matrix from Newton's system
is both symmetrical and tridiagonal.
It is therefore possible to store
the complete system using three vectors
(two for the matrix, and one for the
second member).
However, in practice, to increase efficiency
the solution of the system is calculated
at the same time as the cells are constructed.
This allows to store only one
intermediate vector.
As in classical semi-discrete optimal transport,
the pure Newton's algorithm (i.e.~with
constant step size $\eta^{(k)} \equiv 1$)
may not converge. We therefore use
a backtracking strategy, similar to that 
of~\cite{kitagawa2019convergence},
to ensure that the iterates do not reach
an invalid state (empty cells, negative
potentials, etc.).
The library developed for this publication
is available
at~\url{https://github.com/sdot-team/usdot}.

\subsection{Regularization error}

The impact of regularization on the computed
$\psi$ is illustrated in
Figure~\ref{fig:regularization-error}.
Although this example corresponds to
a specific numerical setting, similar behavior
was consistently observed across all our
experiments: the error scales with the square
of $\eps$ and can easily become negligibly small.
In particular, these simulations
illustrate the tight convergence rate provided
by Theorem~\ref{thm:ODE-and-convergence}.

\begin{figure}
    \centering
    \scalebox{0.7}{
        \input{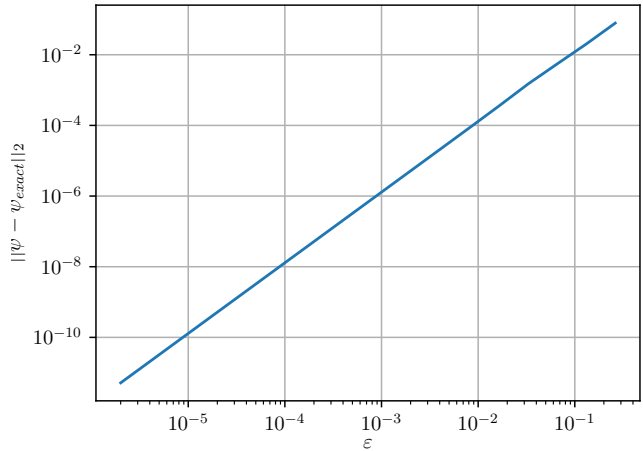}
    }
 
    \caption{Impact of the regularization
    parameter $\eps$ on the potentials $\psi$.
    For this example, the density function was
    a Gaussian centered around 0, with $\sigma=1$,
    the discrete measure is supported on $15$ Dirac
    masses which were positioned randomly in
    $[-1,+1]$ with a uniform distribution, 
    and the ratio of the total mass of the
    Dirac masses to
    that of the Gaussian was set at $1/2$.
    As can be seen, the error on the potentials
    is globally proportional to the square of
    the $\eps$ parameter, in accordance with
    the theorem~\ref{thm:ODE-and-convergence}.} 

    \label{fig:regularization-error}
\end{figure}

\subsection{Shape registration example}

Bonneel et al.~\cite{bonneel2015sliced}
introduced a shape registration algorithm based
on a 1D discrete-discrete optimal transport
algorithm, known as the Fast Iterative Sliced
Transport (FIST) algorithm.
This method leverages
1D point correspondences to generate
unidimensional displacement proposals,
which are then used as combinations to
propose movements in the target space.
Although these proposed displacements may
not perfectly match those obtained through
direct optimal transport in the target space,
the approximation can still effectively
capture global motions, such as rigid body
transformations.

This algorithm has been adapted for the
semi-discrete setting, the main difference
being the representation of the target shape.
In our case, the projections can be defined
as almost any kind of functions, for instance
sum of polynomials, Gaussians, etc.
It is still possible to start from a
set of points, for instance by binning
the projections in histograms, but the
semi-discrete setting offers the possibility
of starting from more generic and more precise
representations. For instance, if the
target shape is defined by a 3D triangular mesh,
one simply has to sum the projection of
the triangles into a piecewise affine function.
One may consider that the triangles have some
thickness to avoid Dirac masses
in borderline cases.

For the semi-discrete setting, the
unidimensional displacement proposals are defined
by the displacements between the Dirac masses and
the barycenter of the corresponding cells.
Figure~\ref{fig:shape-registration} illustrates the
two numerical experiments used to benchmark
the semi-discrete approach for the FIST algorithm. 

\begin{figure}
    \centering
    \includegraphics[width=0.45\textwidth]{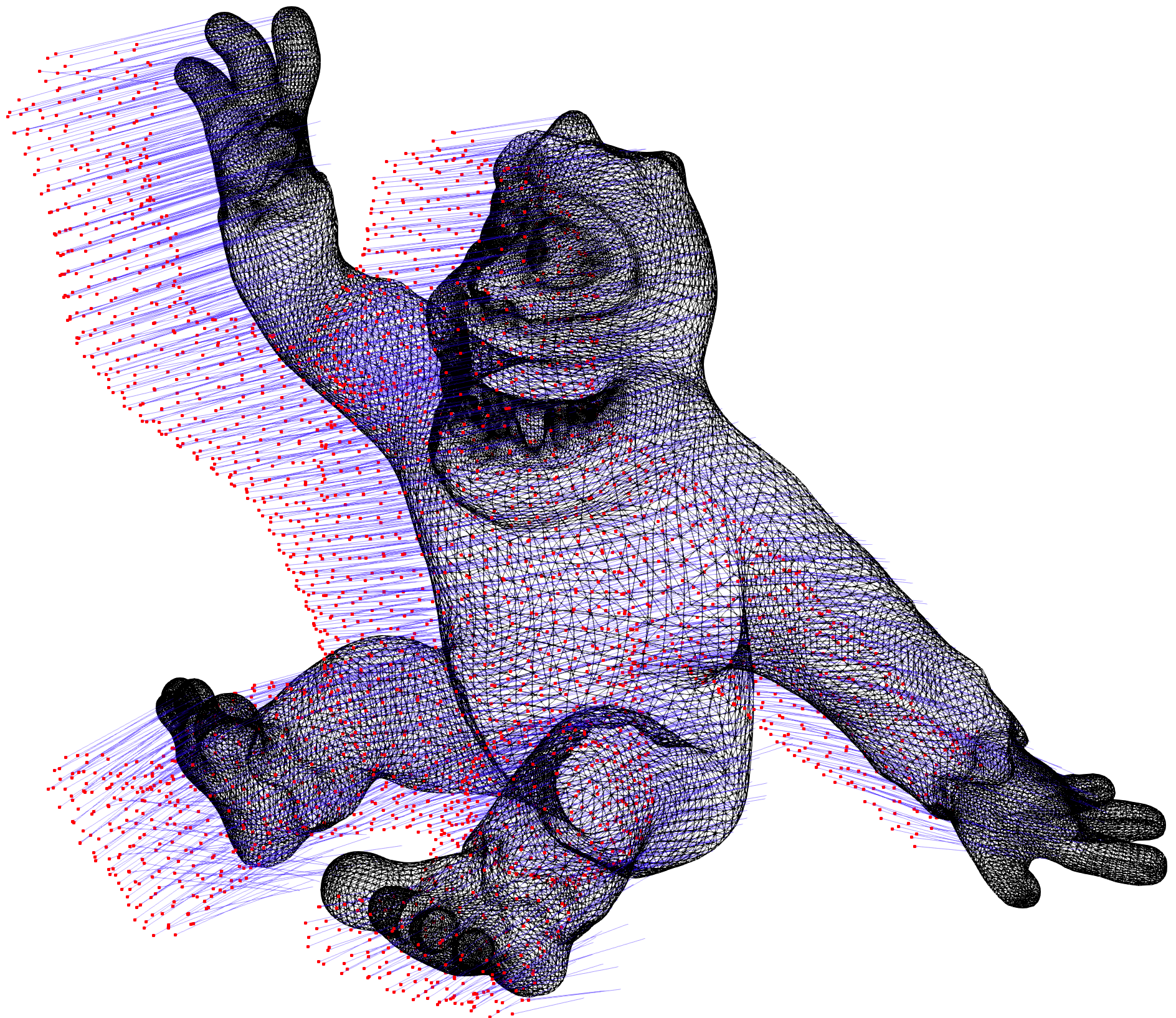}
    \includegraphics[width=0.45\textwidth]{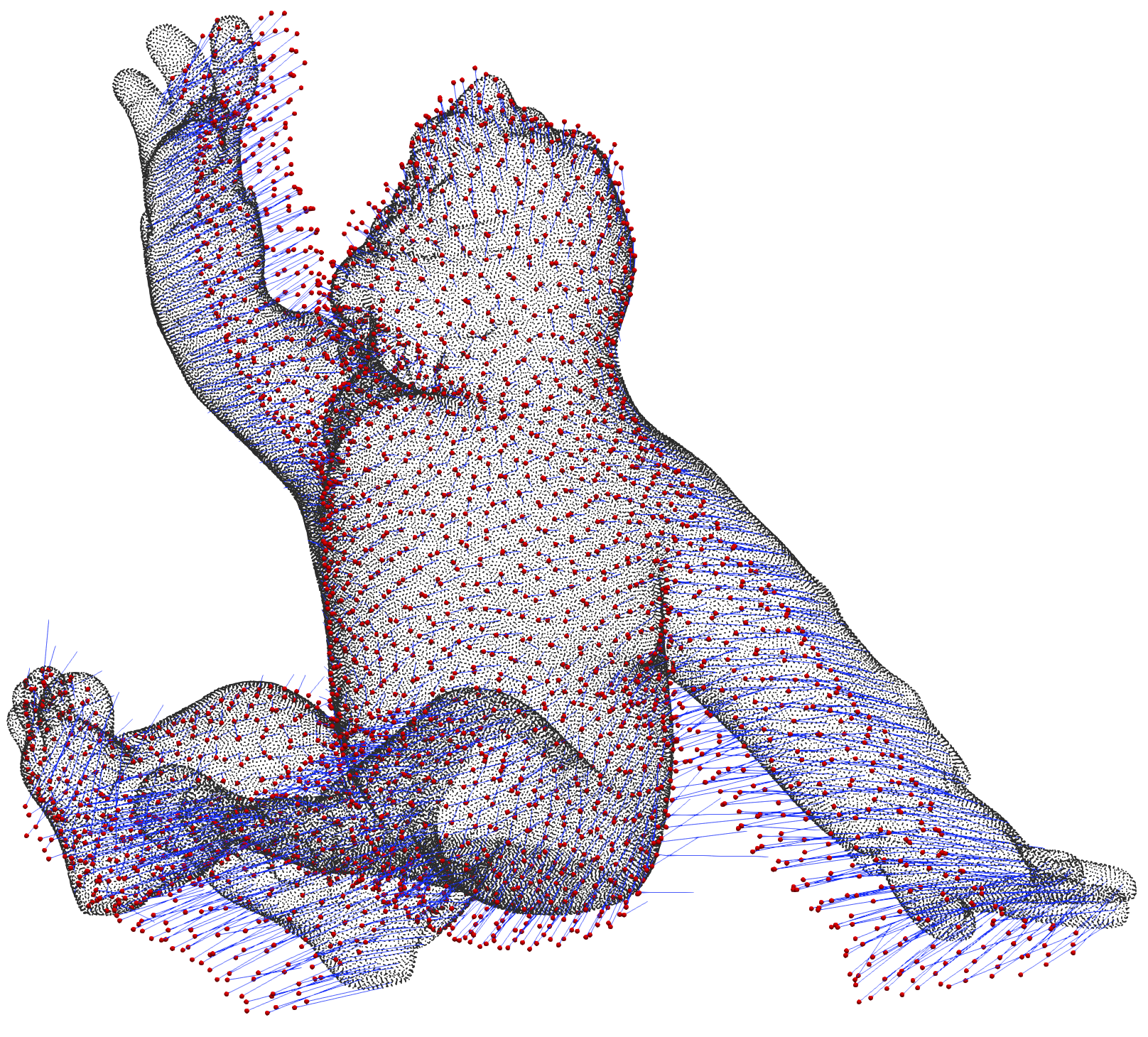}
    
    \caption{Illustration of an iteration
    of the FIST algorithm for two semi-discrete
    settings. For the left image, the
    unidimensional densities are computed using
    the projected mass of the triangles, which
    gives a piecewise affine function.
    For the right image, unidimensional densities
    are histogram of the projected points.
    The black items (lines and dots) represent
    the target shape. The red dots represent
    the position of the Dirac masses
    (the source points).
    The blue lines represent the displacement
    from the Dirac positions to the estimated 3D
    barycenters, that come from the combinations
    of the 1D barycenters with their
    respective directions.}
    
    \label{fig:shape-registration}
\end{figure}

In these examples, computation time is strongly
influenced by the distance to the solution.
Nonetheless, in general, we observe that semi-discrete
problems are approximately 4 to 5 times slower
than fully discrete problems solved using
Bonneel’s algorithm (\cite{bonneel2015sliced}).
This additional cost can be seen as the trade-off
for increased generality and potentially higher
accuracy. However, there is clearly room for
further optimization, particularly to
accelerate convergence.
For instance, in the rigid body registration
examples, the Newton method requires an average
of 63 iterations --- a very high number.
While strategies to reduce this iteration count
have been explored, they have not yet been
implemented within the regularization
framework proposed in this work.

\subsection{Synthetic example}

In this example, we explore the gains in
accuracy that come from being able to use
generic densities, without having to discretize
them into sums of Dirac masses.
The density $\rho$ for this example is equal to
$(1-|x|)_+$. We start with 100 randomly placed
points in the interval $[-1,1]$.
They are associated with uniform weights,
such that the sum is equal to $3/4$
(i.e.~$3/4$ of the mass of $\rho$).
The barycenters are then
computed, with the semi-discrete method,
and with the discrete-discrete method with
a growing number of points for the
discretization of $\rho$.
For the discrete-discrete setting, the spot
library has been used (\url{https://github.com/nbonneel/spot}).
It is highly optimized in terms of execution time.
Nevertheless, the library requires that there be
a correspondence for each Dirac. We therefore
repeated the list of initial Dirac masses in order
to obtain the right mass ratio for the partial
transport problem.

Figure~\ref{fig:barycenter-error} shows the accuracy
losses due to the discretization of $\rho$ used
to reduce to a discrete-discrete problem,
while Figure~\ref{fig:execution-time-DD-vs-SD} shows the relative
timings to compute the barycenters using a
the discrete-discrete and semi-discrete settings, respectively.
As expected, in this case, the discrete-discrete
leads to discretization errors.
Of course, increasing the accuracy leads
to increased execution times, and can quickly
exceed to time needed to solve the exact problem
with a semi-discrete setting.

\begin{figure}
    \centering
    \input{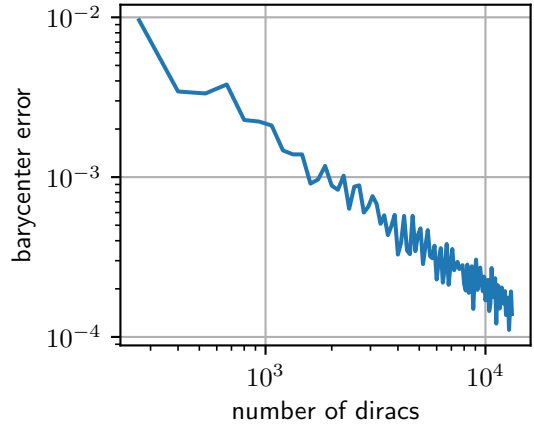}
    
    \caption{Accuracy losses for a
    discrete-discrete setting vs a semi-discrete
    setting for $\rho=(1-|x|)_+$ with 100
    Dirac masses.
    The accuracy is measured as the Euclidean
    distance between the exact barycenters
    (computed using semi-discrete optimal
    transport), and the ones computed using
    the points taken from the discretization
    of $\rho$ with a discrete-discrete setting.}
    
    \label{fig:barycenter-error}
\end{figure}

\begin{figure}
    \centering
    \input{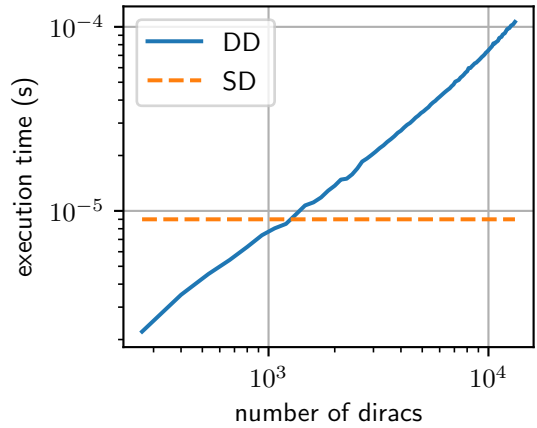}
    
    \caption{Execution time to solve the
    partial optimal transport problem with
    a discrete-discrete (DD) setting and with
    a semi-discrete (SD) setting.
    The Dirac masses for the count in the
    abscissa are the one that are used to
    discretize $\rho$, to be able to use the
    discrete-discrete setting}
    
    \label{fig:execution-time-DD-vs-SD}
\end{figure}

\section*{Acknowledgments}
The authors would like to warmly thank
Quentin Mérigot and Luca Nenna for their
guidance and support in the elaboration
of this work. This work has been partially
supported by the UDOPIA doctoral program,
as well as the Agence nationale de la recherche,
through the ANR project GOTA (ANR-23-CE46-0001)
and the PEPR PDE-AI project (ANR-23-PEIA-0004).

\bibliographystyle{plain}
\bibliography{bibliography}

\appendix

\section{Proof of Theorem~\ref{thm:regularity-multiD}}
\label{sec:appendix-multiD}

Thanks to Proposition~\ref{prop:C1-regularity},
it is enough to prove that the functions
$G_i : \psi \mapsto
\int_{\RLag_i(\psi)}
\rho(x)dx$ are $\Cc^1$ on $\mathcal{D}$.
Notice that $\mathcal{D}$ is both open and convex,
since the inequalities
$\|x-y_i\|^2-\psi_i \leq
\|x-y_j\|^2-\psi_j$ and
$\|x-y_i\|^2\leq \psi_i$
are respectively linear and convex in $(x,\psi)$.
Fix $\psi^0\in\mathcal{D}$
and let $\psi^t = \psi^0 + t v$,
where $v\in\R^N$ is a non-zero vector.
For the sake of readability, we denote
$\mathrm{V}_j^t = \RLag_j(\psi^t)$,
$V_\infty^t=\RLag_\infty(\psi^t)$,
and
$L_j^t=\Lag_j(\psi^t)$.

\subsection*{Polar coordinates}
Let us express the integral in polar coordinates
centered at the Dirac position $y_i$.
Using~\cite[Theorem~2.49]{folland1999real}, we have
\begin{align}
G_i(\psi^t) &=
\int_{\R^d} \mathbbm{1}_{V_i^t}(x)\rho(x) dx
= \int_{\S^{d-1}}\left(\int_0^\infty
\mathbbm{1}_{V_i^t}(s\theta)
\rho(y_i + s\theta)s^{d-1}ds\right)
d\H^{d-1}(\theta)\nonumber\\
\label{eq:integral-over-sphere}
&= \int_{\S^{d-1}}
\left(\int_0^{\omega_{i}^t(\theta)}
\rho(y_i + s\theta)|s|^{d-1}ds\right)
d\H^{d-1}(\theta),
\end{align}
where $\omega_i^t(\theta)$ is the supremum of
the set of real numbers $s\in\R$ for which
$y_i+s\theta \in V_i^t$,
or $0$ if no such $s$ exists.
The last equality is obtained by making the change
of variable $\theta\leftrightarrow-\theta$ in
half of the integral, and summing back the two halves.
Note that equation~\eqref{eq:integral-over-sphere}
holds even when $y_i$ is not in the
interior of $V_i^t$, because the
$\omega_i^t(\theta)$ can be negative.

\subsection*{Cover of the sphere}
We now define a cover of the sphere $\S^{d-1}$
by $N+1$ sets with $\H^{d-1}$-negligible
pairwise intersections. These sets correspond
to the different facets of the boundary
of the cell (some of which may be empty),
except for the last set, which corresponds to the lines
passing through $y_i$ without intersecting the cell.
Namely,
for each $j\in\{1,\dots,N,\infty\}$ distinct from $i$,
we let $\Theta_{i,j}^t$ be the set of unit vectors
$\theta\in\S^{d-1}$ such that
$y_i+\omega_i^t(\theta)\in V_j^t$.
As for $\Theta_{i,i}^t$, we define it as the set of unit
vectors which do not belong to any other $\Theta_{i,j}^t$.
It is immediate to check that
\begin{equation}
\label{eq:def-omega_ij^t-of-theta}
\forall \theta\in\Theta^t_{i,j},
\qquad
\omega_i^t(\theta) =
\begin{cases}
\frac{\|y_j-y_i\|^2 - (\psi_j^t-\psi_i^t)}{
2\langle y_j-y_i, \theta\rangle}
&\qquad\mathrm{if}\quad j\notin\{i,\infty\}
\\
\sqrt{\psi_i^t}
&\qquad\mathrm{if}\quad j=\infty, \\
0 &\qquad\mathrm{if}\quad j=i.
\end{cases}
\end{equation}
The first expression is derived by developing the squares
in the equality
$\|(y_i+\omega\theta)-y_i\|^2-\psi_i^t
=\|(y_j+\omega\theta)-y_j\|^2-\psi_j^t$,
which defines the hyperplane of points
$x=y_i+\omega\theta$ for which traveling to $y_i$
costs exactly as much as traveling to $y_j$.
The same goes for the second expression, except
that we replace the second member of the equality
by zero.

\subsection*{Differentiating under the integral sign}
We now make the temporary additional assumption that
$y_i$ is not contained in any of
the $N$ iso-hypersurfaces defined by
$\underline{c}(x,y_i)-\psi^0
= \underline{c}(x,y_j)-\psi_j^t$,
$j\in\{1,\dots,N,\infty\}\setminus\{i\}$.
Thanks to this assumption,
for $t$ in a small enough neighborhood of zero,
each line passing through $y_i$ intersects the boundary
of the cell in at most two points.
In particular, since the facets of the cell
vary continuously in $t$, every $\theta$
in the relative interior of some $\Theta_{i,j}^0$
is also in $\Theta_{i,j}^t$ for each $t$ in
a small enough neighborhood of $0$
(which depends on $\theta$). As a result,
$t\mapsto \omega_i^t(\theta)$ is differentiable
around $0$ for every such $\theta$, and hence
for $\H^{d-1}$-almost every $\theta\in\S^{d-1}$,
and equation~\eqref{eq:def-omega_ij^t-of-theta}
can be used to find the derivative.
We can therefore differentiate under the integral sign
in~\eqref{eq:integral-over-sphere} and find
\begin{align*}
\partial_{t=0}[G_i(\psi^t)]
&=
\int_{\Theta_i^0} [\partial_t|_{t=0}\omega_i^0(\theta)]
\rho(y_i + \omega_i^0(\theta)\theta)
|\omega_i^0(\theta)|^{d-1} d\H^{d-1}(\theta) \\
&= \sum_{j\in\{1,\dots,N\}\setminus\{i\}}
\frac{-(v_j-v_i)}{2\langle y_j-y_i,\theta\rangle}
\int_{\Theta_{i,j}^0}
\rho(y_i + \omega_{i,j}^0(\theta)\theta)
|\omega_{i,j}^0(\theta)|^{d-1} d\H^{d-1}(\theta) \\
\nonumber
&\quad + \frac{v_i}{2\sqrt{\psi_i^0}}
\int_{\Theta_{i,\infty}^0}
\rho(y_i + \omega_{i,\infty}^0(\theta)\theta)
|\omega_{i,\infty}^0(\theta)|^{d-1} d\H^{d-1}(\theta).
\end{align*}
Since all terms are linear in $v$ and since $v$
was an arbitrary vector, we conclude that
$G_i$ is differentiable at $\psi^0$.
The changes of variables
$\theta\mapsto y_i + \omega_{i,j}^0(\theta)\theta$
allow to recover the expressions given
in~\eqref{eq:non-diag-coeff-Hessian-multiD}
and~\eqref{eq:diag-coeff-Hessian-multiD}
Indeed, thanks to the
additional assumption, these maps are
diffeomorphisms from their respective facet
$V_i^t\cap V_j^t$, $j\neq i$ of the cell onto their
corresponding set $\Theta_{i,j}^t$ of unit vectors.
The continuity in $\psi\in\mathcal{D}$ of the
facet integrals appearing
in~\eqref{eq:non-diag-coeff-Hessian-multiD}
and~\eqref{eq:diag-coeff-Hessian-multiD}
can be established as in the proof
of~\cite[Proposition~B.1]{kitagawa2019convergence}
by Mérigot et al.

Since nonzero mass of every Laguerre cell implies
strict positivity of all the components of the potential,
the $\psi\in\mathcal{D}$ which do not satisfy the additional
assumption are contained in the $N-1$ hyperplanes
of $\R^N$ of equations $\psi_{j}-\psi_{i} = \|y_{j}-y_{i}\|^2$.
By classical arguments, $G_i$ is thus $\Cc^1$
on the whole domain $\mathcal{D}$, with first partial
derivatives given
by~\eqref{eq:non-diag-coeff-Hessian-multiD}
and~\eqref{eq:diag-coeff-Hessian-multiD},
up to a minus sign.
In fact, the singular $\psi$ are an artifact due to
the polar coordinates being centered at $y_i$.
Centering these coordinates at some arbitrary point
in the interior of the cell (e.g. the $\rho$-barycenter of
the cell), one can directly prove differentiability of $G_i$
on all of $\mathcal{D}$. At the cost, of course, of
a less concise expression for $\omega_i^t(\theta)$
when $\theta\in\Theta_{i,\infty}^t$, because
projecting a sphere on a ball is less simple when
they are not concentric.

\subsection*{Strict concavity}
Since $G$ is $\Cc^1$ on $\mathcal{D}$,
the functional $\mathcal{K}$ is $\Cc^2$
on the same domain.
Thanks to~\eqref{eq:non-diag-coeff-Hessian-multiD}
and~\eqref{eq:diag-coeff-Hessian-multiD},
for $\psi\in\mathcal{D}$ the Hessian matrix
$D^2\mathcal K(\psi)$
is irreducible, weakly diagonally dominant,
and has at least one strictly dominant
diagonal coefficient --- take
any cell $\RLag_i$ which has
non-degenerate intersection with
$\spt\rho\setminus
\cup_{j\neq i}\RLag_j$.
The invertibility of $D^2\mathcal K$ then
follows from Taussky's
theorem~\cite[Theorem~II]{horn2012matrix}, and
in particular, the optimal potential is unique.

\section{Proof of Lemma~\ref{lem:technical-SDOPT-1D}}
\label{sec:appendix-uniqueness-SDOPT-1D}

As explained in section~\ref{subsec:SDOPT-1D},
for $\psi\in\mathcal{D}$,
the (unrestricted) Laguerre cells write
$\Lag_i = [z_{i-1},z_i]$, where
\begin{equation*}
z_i(\psi) =
\begin{cases}
-\infty \quad &\mathrm{if} \quad i = 0, \\
\frac{(y_{i+1}^2-\psi_{i+1})-(y_i^2-\psi_i)}
{2(y_{i+1} - y_{i})}
\quad &\mathrm{if} \quad
0 < i < N, \\
+\infty \quad &\mathrm{if} \quad i = N,
\end{cases}
\end{equation*}
and we can therefore express the
\textit{restricted} Laguerre cells as
$\RLag_i = [a_i,b_i]$, where
\begin{equation*}
\begin{cases}
a_i(\psi) = \max\{z_{i-1}(\psi),
&y_i - \sqrt{\psi_i}\},\\
b_i(\psi) = \min\{z_{i}(\psi),
&y_i + \sqrt{\psi_i}\}.
\end{cases}
\end{equation*}

\begin{lemma}
\label{lem:directional-derivatives-1D}
Fix $\psi \in \mathcal D$
and let $v \in \R^N$ be non-zero. Then
\begin{equation}
\label{eq:directional-derivative-of-a}
\partial_v^+ a_i(\psi)
= \begin{cases}
\frac{-(v_i-v_{i-1})}{2(y_i-y_{i-1})}
&\quad \mathrm{if} \quad z_{i-1}(\psi)
> y_i - \sqrt{\psi_i}, \\
\frac{-v_i}{2\sqrt{\psi_i}}
&\quad \mathrm{if} \quad z_{i-1}(\psi)
< y_i - \sqrt{\psi_i}, \\
\max\left\{
\frac{-(v_i-v_{i-1})}{2(y_i-y_{i-1})},
\frac{-v_i}{2\sqrt{\psi_i}}
\right\}
&\quad \mathrm{if} \quad z_{i-1}(\psi)
= y_i - \sqrt{\psi_i},
\end{cases}
\end{equation}
and
\begin{equation}
\label{eq:directional-derivative-of-b}
\partial_v^+ b_i(\psi)
= \begin{cases}
\frac{v_i}{2\sqrt{\psi_i}}
&\quad \mathrm{if} \quad z_{i}(\psi)
> y_i + \sqrt{\psi_i}, \\
\frac{-(v_{i+1}-v_i)}{2(y_{i+1}-y_i)}
&\quad \mathrm{if} \quad z_{i}(\psi)
< y_i + \sqrt{\psi_i}, \\
\min\left\{
\frac{-(v_{i+1}-v_i)}{2(y_{i+1}-y_i)},
\frac{v_i}{2\sqrt{\psi_i}}
\right\}
&\quad \mathrm{if} \quad z_{i}(\psi)
= y_i + \sqrt{\psi_i}.
\end{cases}
\end{equation}
Moreover, the unilateral directional
derivative of the mass of the $i$\textsuperscript{th}
restricted Laguerre cell is
\begin{equation}
\label{eq:directional-derivative-of-G}
\partial_v^+G_i =
\rho(b_i) \partial_v^+b_i
- \rho(a_i) \partial_v^+a_i,
\end{equation}
so we can write
$\partial_v^+G(\psi) = H(\psi,v)v$
where $H(\psi,v) \in \R^{N\times N}$
is a tridiagonal matrix whose rows are all
weakly diagonally dominant.
\end{lemma}

There may be several possibilities
for the coefficients of $H(\psi,v)$,
so we fix  coefficients to be those which are
directly readable from~\eqref{eq:directional-derivative-of-a}
and~\eqref{eq:directional-derivative-of-b}.
A case where there might be an ambiguity is
for instance the case where
$z_{i-1}(\psi) = y_i - \sqrt{\psi_i}$
and where the two quantities in the maximum
in~\eqref{eq:directional-derivative-of-a} are equal,
in which case  we choose
to read the coefficients from
the quantity on the right in the maximum.
We use the same convention for the
other case of ambiguity, namely
when $z_i(\psi) = y_i + \sqrt{\psi_i}$
and both quantities in the minimum
in~\eqref{eq:directional-derivative-of-b} are equal,
in which case we read the coefficients from
the second argument of the minimum.

\begin{proof}
The expressions of
$\partial_v^+a_i(\psi)$
and $\partial_v^+b_i(\psi)$
come from the fact that
\begin{equation*}
\partial_v^+z_i(\psi)
= \frac{-(v_{i+1}-v_i)}{2(y_{i+1}-y_i)}
\qquad \mathrm{and} \qquad
\partial_v^+\sqrt{\psi_i}
= \frac{v_i}{2\sqrt{\psi_i}}.
\end{equation*}
As for the expression of $\partial_v^+G_i$,
it follows directly from the fact that
$G_i = F_\rho(b_i) - F_\rho(a_i)$,
where $F_\rho$ is the cumulative probability
distribution of $\rho$.
At last,
from~\eqref{eq:directional-derivative-of-a},
~\eqref{eq:directional-derivative-of-b},
and~\eqref{eq:directional-derivative-of-G},
we deduce that $H(\psi, v)$ is tridiagonal,
with weakly diagonally dominant rows.
\end{proof}

\begin{lemma}
\label{lem:block-decomposition}
Fix $\psi \in \mathcal D$
and let $v \in \R^N$ be non-zero.
Then the matrix $H(\psi,v)$ defined in
Lemma~\ref{lem:directional-derivatives-1D}
is symmetric. As a result, it writes
\begin{equation}
\label{eq:block-decomposition}
H(\psi,v) =
\begin{pmatrix}
    H_{1} & & \\
    & \ddots & \\
    & & H_{r}
\end{pmatrix}
\end{equation}
where each $H_j$ is a symmetric, tridiagonal,
irreducible, diagonally dominant matrix.
Of course, the right-hand side also
depends on $\psi$ and $v$,
but we do not mention it in the notation
for the sake of readability.
\end{lemma}

\begin{proof}
Since $H(\psi,v)$ is tridiagonal,
it suffices to prove that its subdiagonal and
superdiagonal are equal.
We thus fix an index $1\leq i \leq N-1$.
Since the restricted Laguerre cells
of index $i$ and $i+1$ are non-empty,
we must have
$y_i - \sqrt{\psi_i} < z_i(\psi)$ and
$z_i(\psi) < y_{i+1} + \sqrt{\psi_{i+1}}$.
Denoting $\delta = y_{i+1}-y_i$, we have
\begin{equation}
\begin{cases}
z_i(\psi) - (y_i - \sqrt{\psi_i})
&= \frac{1}{2\delta} \left[
\delta^2 + 2\sqrt{\psi_i} \delta
- (\psi_{i+1}-\psi_i) \right], \\
(y_{i+1} + \sqrt{\psi_{i+1}}) - z_i(\psi)
&= \frac{1}{2\delta} \left[
\delta^2 + 2\sqrt{\psi_{i+1}} \delta
+ (\psi_{i+1}-\psi_i) \right],
\end{cases}
\end{equation}
and the roots of the two left-hand side polynomials
are respectively
$-\sqrt{\psi_i}\pm\sqrt{\psi_{i+1}}$ and
$\pm\sqrt{\psi_i}-\sqrt{\psi_{i+1}}$,
so since the right-hand sides are positive we deduce
\mbox{$|\sqrt{\psi_{i+1}}-\sqrt{\psi_i}|<\delta$}.

Similarly,
\begin{equation}
\begin{cases}
z_i(\psi) - (y_i + \sqrt{\psi_i})
&= \frac{1}{2\delta}\left[
\delta^2 - 2\sqrt{\psi_i}\delta
- (\psi_{i+1}-\psi_i) \right], \\
(y_{i+1}-\sqrt{\psi_{i+1}}) - z_i(\psi)
&= \frac{1}{2\delta}\left[
\delta^2 - 2\sqrt{\psi_{i+1}}\delta
+ (\psi_{i+1}-\psi_i) \right],
\end{cases}
\end{equation}
and the roots of the left hand-sides
are respectively
$\sqrt{\psi_i}\pm\sqrt{\psi_{i+1}}$ and
$\pm\sqrt{\psi_i}+\sqrt{\psi_{i+1}}$,
so unless $\delta = \sqrt{\psi_i}+\sqrt{\psi_{i+1}}$
the two polynomials are non-zero and have
the same sign. In other words, we have either
\begin{equation}
\label{eq:case1-cells-not-touching-1D}
y_i + \sqrt{\psi_i} < z_i(\psi)
< y_{i+1} - \sqrt{\psi_{i+1}}
\end{equation}
or
\begin{equation}
\label{eq:case2-cells-touching-1D}
y_{i+1} - \sqrt{\psi_{i+1}}
< z_i(\psi) < y_i + \sqrt{\psi_i}.
\end{equation}

\begin{enumerate}[label=\textbf{Case \arabic*}]
\item
\label{item:case1-cells-not-touching-1D}
If~\eqref{eq:case1-cells-not-touching-1D} holds,
then for any $\psi'$ close enough to $\psi$,
$b_i$ does not depend on $\psi'_{i+1}$
and $a_{i+1}$ does not depend on $\psi'_i$,
so $H_{i+1,i}(\psi,v)
= H_{i,i+1}(\psi,v) = 0$.

\item
\label{item:case2-cells-touching-1D}
If~\eqref{eq:case2-cells-touching-1D} holds,
then on a neighborhood of $\psi$ we have
$b_i = a_{i+1} = z_i$,
and so a straightforward computation yields
$H_{i+1,i}(\psi,v) = H_{i,i+1}(\psi,v) =
-\rho(z_i(\psi))
\frac{1}{2(y_{i+1}-y_i)}$,
thanks to the expression of $z_i$.

\item
\label{item:case3-degenerate-1D}
Let us now handle the degenerate case where
$\delta = \sqrt{\psi_i} + \sqrt{\psi_{i+1}}$.
Then $z_i(\psi) = y_i + \sqrt{\psi_i}
= y_{i+1} - \sqrt{\psi_{i+1}}$,
which implies that $b_i(\psi)=a_{i+1}(\psi)$,
and thanks to~\eqref{eq:extremities-of-RLag-1D}
we find
\begin{equation*}
\partial_v^+b_{i}(\psi)
= \min\left\{
\frac{v_i-v_{i+1}}{2(
\sqrt{\psi_i} +
\sqrt{\psi_{i+1}})},
\frac{v_i}{2\sqrt{\psi_i}}
\right\}
\end{equation*}
and
\begin{equation*}
\partial^+_v a_{i+1}(\psi)
= \max\left\{
\frac{v_{i}-v_{i+1}}{2(
\sqrt{\psi_i} +
\sqrt{\psi_{i+1}})
)},
\frac{-v_{i+1}}{2\sqrt{\psi_{i+1}}}
\right\}.
\end{equation*}
To conclude, we notice that
\begin{align*}
\frac{v_{i}-v_{i+1}}{2(
\sqrt{\psi_i} +
\sqrt{\psi_{i+1}})}
< \frac{v_i}{2\sqrt{\psi_i}}
&\qquad \Longleftrightarrow \qquad
0 < \frac{v_i}{\sqrt{\psi_i}}
+ \frac{v_{i+1}}{\sqrt{\psi_{i+1}}} \\
&\qquad \Longleftrightarrow \qquad
\frac{v_{i}-v_{i+1}}{2(
\sqrt{\psi_i} + \sqrt{\psi_{i+1}})}
> \frac{-v_{i+1}}{2\sqrt{\psi_{i+1}}},
\end{align*}
and this series of equivalence still holds if
all strict inequality signs are reversed.
As a consequence, the coefficient of
$v_{i+1}$ in $\partial_v^+b_i(\psi)$
is the opposite of that of
$v_i$ in $\partial_v^+a_{i+1}(\psi)$,
and so once again
$H_{i+1,i}(\psi,v)=H_{i,i+1}(\psi,v)$.
\end{enumerate}
\end{proof}

\begin{remark}
\label{rem:3-cases-1D}
In the proof of Lemma~\ref{lem:block-decomposition},
case~\ref{item:case1-cells-not-touching-1D}
corresponds to the situation where
$\RLag_i$ and $\RLag_{i+1}$ touch.
On the contrary,
case~\ref{item:case2-cells-touching-1D} is
the situation where they touch and the
inequality $z_i<y_i+\sqrt{\psi_i}$ is strict.
At last, case~\ref{item:case3-degenerate-1D}
is the degenerate situation where the
two cells touch but $z_i=y_i+\sqrt{\psi_i}$,
so that in any neighborhood of $\psi$
there is a $\psi'$ such that they do not touch
and another such that they do.
\end{remark}

\begin{lemma}
\label{lem:strictly-dominant-coefficient}
Each $H_j$ in the
decomposition~\eqref{eq:block-decomposition}
has a strictly dominant diagonal coefficient.
\end{lemma}
\begin{proof}
We recall that
\begin{equation*}
\partial_v^+G_i =
\rho(b_i)\left[
\left.
\frac{-(v_{i+1}-v_i)}{2(y_{i+1}-y_i)}
\right\vert
\frac{v_i}{2\sqrt{\psi_i}}
\right]
+
\rho(a_i)\
\left[
\left.
\frac{v_i-v_{i-1}}{2(y_i-y_{i-1})}
\right\vert
\frac{v_i}{2\sqrt{\psi_i}}
\right]
\end{equation*}
where the notation $[c|d]$ stands
for either expression $c$ or expression $d$.
As a result, line $i$ has a strictly dominant
diagonal coefficient if and only if
one of $\rho(a_i)$ or $\rho(b_i)$
is non-zero and the square bracket coefficient
next to it takes the value on the right.

Let us first consider a line of $H(\psi,v)$
with index $1 < i < N$.
Then $\rho(a_i) > 0$ and $\rho(b_i) > 0$,
since the density is strictly positive in the
interior of $\spt\rho$,
which is convex.
As a result, line $i$ is strictly diagonally
dominant if and only if its
$(i-1)$\textsuperscript{th} or
$(i+1)$\textsuperscript{th} coefficient
(or both) is zero.
This ensures that if there are two blocks $H_j$
or more, each of these blocks has a strictly
diagonally dominant line --- its first
or its last.

Let us now tackle the case where the
decomposition consists of a single block, that is,
$H(\psi,v)$ is already irreducible.
Then the restricted Laguerre cells are all
packed together: $b_i = a_{i+1}$
for every $1 \leq i < N$.
Since the complement of their union
--- the additional cell,
associated to the auxiliary point introduced
in Remark~\ref{rem:auxiliary-point} ---
has non-zero $\rho$-measure, we must have
$\min(\spt\rho) < a_1$ or
$b_N < \max(\spt\rho)$.
As a result, the density is non-zero at one
of $a_1$ or $b_N$, so at least one of the
first and last lines is strictly diagonally
dominant.
\end{proof}


\end{sloppypar}
\end{document}